\documentclass[reqno]{amsart}
\usepackage[colorlinks=true, pdfstartview=FitV, linkcolor=blue, citecolor=blue]{hyperref}
\usepackage{amssymb,amsmath,amsfonts}
\usepackage{bm}
\usepackage{enumitem}
\usepackage{array}
\usepackage{graphicx}
\allowdisplaybreaks

\newcommand{\N}{\mathbb{N}}  \newcommand{\R}{\mathbb{R}}  
\newcommand{\Z}{\mathbb{Z}}  \newcommand{\Q}{\mathbb{Q}}  
\newcommand{\rar}{\rightarrow}
\newcommand{\mc}{\mathcal}   \newcommand{\mr}{\mathrm}
\newcommand{\balph}{{\bm{\alpha}}}
\newcommand{\bn}{{\bm{n}}}

\newtheorem*{theorem*}{Theorem}
\newtheorem{theorem}{Theorem}
\newtheorem{lemma}[theorem]{Lemma}
\newtheorem{corollary}[theorem]{Corollary}
\newtheorem{proposition}[theorem]{Proposition}

\begin{document}
\title[Non-rectifiable cut-and-project sets]{A cut-and-project set which is not bi-Lipschitz equivalent to a lattice}
\author{Alan Haynes}



\begin{abstract}
We construct a cut-and-project set in a totally irrational plane in $\R^3$, with a compact window which is the closure of its interior and has countable boundary, which is not bi-Lipschitz equivalent to a lattice. This answers a question of Burago and Kleiner about cut-and-project sets, and shows in particular that the hypothesis of boundary measure zero is not enough to guarantee bi-Lipschitz equivalence to a lattice. Since our set has uniform density, it also shows that a non-rectifiable Delone set need not encode a non-realisable density, answering a question of Bhat and Dymond.
\end{abstract}

\maketitle

\section{Introduction}\label{sec.Intro}

A Delone set in $\R^d$ is said to be \textit{rectifiable} if it is bi-Lipschitz (BL) equivalent to $\Z^d$, that is, if there is a bijection from the set to $\Z^d$ which is Lipschitz and has a Lipschitz inverse. Mikhael Gromov \cite{Grom1993}, motivated by quasi-isometries of groups, and Hillel Furstenberg, motivated by Kakutani equivalence of $\R^d$-actions, asked whether every Delone set in $\R^d$ is rectifiable (see \cite[Section 3]{SmilSolo2026} for the history). For $d=1$ the answer is yes, for elementary reasons. In 1998 Dmitri Burago and Bruce Kleiner \cite{BuraKlei1998} and, independently, Curtis McMullen \cite{McMu1998} showed that for $d\ge 2$ the answer is no. Both constructions begin with a bounded density $\rho$ on the unit square which is not the Jacobian of any BL homeomorphism, and then build a Delone set whose rescaled restrictions to a sequence of larger and larger squares converge to $\rho$.

Sets produced in this way are deliberately irregular, and much of the work since then has concerned Delone sets coming from aperiodic order, which are described by a finite amount of data. Every Delone set arising from a primitive substitution tiling is rectifiable \cite{Solo2011}, and so is every linearly repetitive Delone set \cite{AlisCoroGamb2013}. On the other hand, Cortez and Navas \cite{CortNava2016} constructed repetitive subsets of $\Z^2$ which are not rectifiable, and Bhat and Dymond \cite{BhatDymo2026} found such sets with close to optimal repetitivity functions. In both cases non-rectifiability comes from a version of the Burago--Kleiner encoding of a non-realisable density.

The question of whether every cut-and-project set (also called a model set) is BL equivalent to a lattice has remained open. Burago and Kleiner \cite[Introduction, Problem 3]{BuraKlei2002} considered the following construction: let $P\subseteq\R^n$ be a $2$-plane with irrational slope, let $B\subseteq\R^n$ be a bounded set with nonempty interior, and project orthogonally to $P$ the set of points $z\in\Z^n$ for which $(z+B)\cap P\not=\emptyset$. They asked whether the resulting net is always BL equivalent to $\Z^2$, and the proof of \cite[Theorem 1.2]{BuraKlei2002} settles the case $n=3$ for planes with `Diophantine slope' (condition (1.1) of \cite{BuraKlei2002}), under the condition that the orthogonal projection of $B$ to the normal line of $P$ is a finite union of intervals. We note that there appears to be a discrepancy between what is claimed in \cite{BuraKlei2002} and what follows from the machinery of the proof: their theorem is stated for arbitrary bounded $B$ with nonempty interior, but the proof uses the fact that the set of lattice points which are projected is the intersection of $\Z^3$ with a slab parallel to $P$. This is the case when the projection of $B$ to the normal line of $P$ is an interval, as happens for example when $B$ is connected. The density estimates in the proof are additive, and uniform in the position and width of the slab, so the same argument applies to the intersection of $\Z^3$ with a finite union of such slabs. We point this out because our example falls within the scope of the theorem as stated (see the discussion after Theorem \ref{thm.Main}).

Haynes, Kelly, and Weiss \cite[Theorem 1.1]{HaynKellWeis2014} showed, in all dimensions, that for almost every subspace the answer to the BL equivalence question is yes, whenever the boundary of the window has upper Minkowski dimension smaller than the dimension of the internal space. Later, Koivusalo and Lagac\'{e} \cite{KoivLaga2025} found sharp bounds for the discrepancy of cut-and-project sets, which indicate that this is close to the largest class to which the known sufficient criteria for rectifiability apply.

The BL problem for cut and project sets appears again, for windows which are bounded, have nonempty interior, and have boundary of measure zero, as Problem 2.8.2 in \cite{AdicEtal2016}, where it is pointed out that if the window is only required to be bounded then it is easy to produce non-rectifiable examples. Most recently it is listed by Smilansky and Solomon as Open Problem 7.2 in their survey \cite{SmilSolo2026}. To our knowledge, no example of a non-rectifiable cut-and-project set with window satisfying these conditions was previously known.

To state our main result, for $\balph=(\alpha_1,\alpha_2)\in\R^2$ let
\[
V_\balph=\{z\in\R^3 : z_3=\alpha_1z_1+\alpha_2z_2\},\qquad \bn_\balph=(\alpha_1,\alpha_2,-1),
\]
and let $\pi_\balph$ denote the orthogonal projection from $\R^3$ onto $V_\balph$. For a bounded set $W\subseteq\R$ we write
\begin{equation}\label{eqn.CPSDef}
Y_\balph(W)=\left\{\pi_\balph(z) : z\in\Z^3,~\langle z,\bn_\balph\rangle\in W\right\}.
\end{equation}
This is the cut-and-project set with physical space $V_\balph$, internal space $V_\balph^\perp$ (identified with $\R$ by the map $w\mapsto\langle w,\bn_\balph\rangle$), lattice $\Z^3$, and window $W$.

\begin{theorem}\label{thm.Main}
There exist $\alpha_1,\alpha_2\in\R$, with $1,\alpha_1,$ and $\alpha_2$ linearly independent over $\Q$, and a compact set $W\subseteq\R$, which is the closure of its interior and has countable boundary, with the following properties:
\begin{itemize}[itemsep=5bp,parsep=10bp, topsep=0bp]
\item[(i)] There is no $z\in\Z^3$ with $\langle z,\bn_\balph\rangle\in\partial W$, and
\item[(ii)] The set $Y_\balph(W)$ is a Delone set in $V_\balph$ which is not BL equivalent to any lattice in $V_\balph$.
\end{itemize}
\end{theorem}

Since $1,\alpha_1,$ and $\alpha_2$ are linearly independent over $\Q$, the plane $V_\balph$ is totally irrational in each of the senses in use: it contains no nonzero rational vector, its orthogonal complement contains no nonzero integer vector% (if $t\bn_\balph\in\Z^3$ then $t\in\Z$, from the third coordinate, and $t\alpha_1\in\Z$, so $t=0$)
, and its image in $\R^3/\Z^3$ is dense. Consequently $\pi_\balph$ is injective on $\Z^3$, and, since $\alpha_1$ is irrational, the internal projection of $\Z^3$, which under our identification is the set $\Z\alpha_1+\Z\alpha_2+\Z$, is dense in $\R$. The window is compact, is the closure of its interior, and has boundary of measure zero, and by part (i) no lattice point projects to its boundary. In the terminology of \cite{BaakGrim2013}, $Y_\balph(W)$ is a regular model set. In particular it satisfies all of the hypotheses of \cite[Problem 2.8.2]{AdicEtal2016}, so Theorem \ref{thm.Main} gives a negative answer to that problem when the ambient space has dimension $3$ and the physical space has dimension $2$. It also answers Open Problem 7.2 of \cite{SmilSolo2026}. To see that it answers Burago and Kleiner's question for $n=3$, take $B=\{-s\bn_\balph/|\bn_\balph|^2 : s\in W\}+D$, where $D$ is the closed unit disc in $V_\balph$ centered at $0$. This is a compact set with nonempty interior, and for $z\in\Z^3$ we have that $(z+B)\cap V_\balph\not=\emptyset$ if and only if $\langle z,\bn_\balph\rangle\in W$, so the net constructed from $B$ is $Y_\balph(W)$. Under the identification of $V_\balph^\perp$ with $\R$ above, the orthogonal projection of this set $B$ to the normal line of $V_\balph$ is $\{-s:s\in W\}$, which has infinitely many components (this is established in Section \ref{sec.Proof}, at the beginning of the proof of part (ii) of Theorem \ref{thm.Plane}).

Our construction works for every irrational $\alpha_1$, and in particular we may take $\alpha_1=\sqrt{2}$, which satisfies condition (1.1) of \cite{BuraKlei2002} (in our notation, condition (1.1) concerns $\alpha_1$ only, and \cite[Theorem 1.2]{BuraKlei2002} imposes no condition on $\alpha_2$). Therefore Theorem \ref{thm.Main} also shows that the property of $B$ used in the proof of \cite[Theorem 1.2]{BuraKlei2002}, that its projection to the normal line of $P$ is a finite union of intervals, cannot be dropped.

The window in Theorem \ref{thm.Main} cannot be replaced by an interval. In a companion paper \cite{Hayn2026} we show that if $W$ is an interval of length less than $1$, and if $\alpha_1$ and $\alpha_2$ are not both rational, then $Y_\balph(W)$ is BL equivalent to $\Z^2$. We do not know whether the window can be replaced by a finite union of intervals. By the discussion above this is impossible when $\alpha_1$ satisfies condition (1.1) of \cite{BuraKlei2002}, so any such example would have to use an $\alpha_1$ which does not.

Theorem \ref{thm.Main} will be deduced in Section \ref{sec.NormalForm} from Theorem \ref{thm.Plane}, which is a statement about the set of integer points $m\in\Z^2$ for which $\alpha_1m_1+\alpha_2m_2$ lies in $W+\Z$. Working with this subset of $\Z^2$ has the advantage that the internal coordinate of each point is given by an explicit linear form, and all of the constructions in the paper are carried out there. The idea of passing to such a set already appears in the author's earlier paper \cite{Hayn2016} on codimension one cut-and-project sets. There, after moving the window along the physical space to a segment parallel to a coordinate axis, the problem was reduced to one about a set of integer points of exactly this form, which was then analyzed one row at a time, through the return times of an irrational rotation to the window.

Our set also has a property which bears on a different question. Following \cite{BhatDymo2026} (see also \cite[Section 3.3]{SmilSolo2026}), a measurable function $\rho:[0,1]^2\rar(0,\infty)$ is \textit{non-realisable} if there is no map $\Phi:[0,1]^2\rar\R^2$, which is BL onto its image, with $|\det D\Phi|=\rho$ almost everywhere. A Delone set $\Lambda\subseteq\R^2$ \textit{encodes} $\rho$ if there are numbers $r_k\rar\infty$, closed squares $Q_k\subseteq\R^2$ of side length $r_k$, and affine bijections $\varphi_k:[0,1]^2\rar Q_k$ such that the measures
\[
\mu_k=\frac{1}{r_k^2}\sum_{x\in\Lambda\cap Q_k}\delta_{\varphi_k^{-1}(x)}
\]
converge weakly to $\rho(u)\,\mr{d}u$. A Delone set which encodes a non-realisable density is not rectifiable (this is the content of \cite[Lemma 2.1]{BuraKlei1998}; see \cite[Proposition 3.13]{SmilSolo2026}), and this is the mechanism behind all of the examples of non-rectifiable Delone sets mentioned above. Bhat and Dymond \cite{BhatDymo2026} asked whether the converse holds, that is, whether every non-rectifiable Delone set encodes a non-realisable density, and this question is recorded as Open Problem 3.14 in \cite{SmilSolo2026}. In what follows, the area of a square in $V_\balph$ is measured with respect to the Euclidean metric inherited from $\R^3$, and $V_\balph$ is identified with $\R^2$ by any linear isometry.

\begin{corollary}\label{cor.Encode}
Let $\balph$ and $W$ be as in Theorem \ref{thm.Main}, and let $\delta=|W|/|\bn_\balph|$. Then, for every $\epsilon>0$, there is an $R_\epsilon>0$ such that every closed square $Q\subseteq V_\balph$ of side length $R\ge R_\epsilon$ satisfies
\[
\left|\#\left(Y_\balph(W)\cap Q\right)-\delta R^2\right|\le\epsilon R^2.
\]
Consequently, up to equality almost everywhere, the only density encoded by $Y_\balph(W)$ is the constant function $\delta$, and in particular $Y_\balph(W)$ does not encode any non-realisable density.
\end{corollary}

The constant function $\delta$ is the Jacobian of the dilation $u\mapsto\delta^{1/2}u$, so it is realisable, and Corollary \ref{cor.Encode} therefore gives a negative answer to the question of Bhat and Dymond. We point out that Burago and Kleiner remark in \cite[Remark 1.5]{BuraKlei2002}, without giving details, that the technique of \cite{BuraKlei1998} can be used to produce separated nets which have uniform asymptotic density but are not BL equivalent to $\Z^2$. Uniform asymptotic density in their sense is the first statement of Corollary \ref{cor.Encode}, and by the second part of the proof of that corollary a net with this property does not encode any non-realisable density. To our knowledge, however, no proof of their remark has appeared, and the question is recorded as open in \cite{SmilSolo2026}. The uniformity of the density is a general feature of regular model sets (cf. \cite{Mood2002}), and in our setting it follows from an elementary counting lemma along rows of lattice points (Lemma \ref{lem.RowCount}), which we need anyway for the proof of Theorem \ref{thm.Main}.
%
%Corollary \ref{cor.Encode} also explains why the proof of Theorem \ref{thm.Main} cannot follow the argument of Burago and Kleiner to the end. Their argument rescales the set by factors tending to infinity, passes to a limit, and obtains a contradiction from the limiting density, which rules out BL maps with every Lipschitz constant at once. For $Y_\balph(W)$ every such limit is constant, so the obstruction to a BL map has to be found at finite scales. This is not special to our window. Whenever $\partial W$ has measure zero, the same argument as in the proof of Corollary \ref{cor.Encode} (cf. \cite{Mood2002}) shows that a square of side $R$ contains $\delta R^2+o(R^2)$ points of $Y_\balph(W)$, uniformly in the position of the square. Consequently, if the counts in two adjacent squares of side $R$ differ by at least $cR^2$ (we refer to $c$ as the \textit{contrast}), then $R$ is bounded in terms of $c$. Since the configurations which rule out $L$-BL bijections only work at scales which are large in terms of $L$ and $c$ (Proposition \ref{prop.Finite} below), the contrast cannot be fixed in advance. In our construction each $L$ is dealt with by a separate configuration, at its own scale, with a contrast $c_L$ which tends to $0$ as $L\rar\infty$.

The first main ingredient of the proof is a finite-scale version of the argument of Burago and Kleiner \cite[Section 3]{BuraKlei1998} based on stretched pairs of points (Proposition \ref{prop.Finite}). Some modification of their argument is needed, since it obtains a contradiction from a limit of rescaled copies of the set, and by Corollary \ref{cor.Encode} every such limit of our set has constant density (this is presumably the situation which Burago and Kleiner have in mind in \cite[Remark 1.5]{BuraKlei2002}). For each $L\ge 1$ and each $c>0$ there is a finite configuration of rectangles, with finitely many levels, each level consisting of checkerboards of squares placed along the edges of a grid in the squares of the previous level, such that the following holds at every sufficiently large scale: a Delone set with covering radius at most $L$ whose counts on adjacent squares of the configuration differ by at least $c$ times their area admits no $L$-BL bijection to $\Z^2$. We refer to the number $c$ here as the \textit{contrast}. We deduce this from a statement about BL homeomorphisms of the plane, in which the hypothesis of Burago and Kleiner on the Jacobian is replaced by a lower bound for the differences between the areas of the images of adjacent squares.

The second ingredient is the window. Since the internal space is one-dimensional, windows of a particular form correspond to simple patterns in the physical space. For an integer $Q\ge 1$, the set of $x\in\R/\Z$ with $\|Qx-\gamma\|<w/2$, which is a union of $Q$ equally spaced arcs, picks out the lattice points in a family of parallel strips in $\R^2$, with normal vector $Q\balph-\bm{\mu}$ for any $\bm{\mu}\in\Z^2$, and unions and intersections of such windows correspond to unions and intersections of the strip families (Lemma \ref{lem.Stage}). We fix $\alpha_1$ irrational and construct $\alpha_2$ by nested intervals, so that the normals of finitely many families approximate prescribed vectors with prescribed accuracy, and this allows us to reproduce the configuration of Proposition \ref{prop.Finite} for any given $L$. The contrast is made small, of order $|A_L|$, by intersecting with a short interval $A_L$, and the intervals $A_L$ shrink to a single point as $L\rar\infty$, which is what makes the boundary of the window countable. The main difficulty is then to show that the patterns for different values of $L$ do not interfere with each other. The patterns which have already been built when the pattern for $L$ is introduced are handled by a uniform counting estimate which depends only on $\alpha_1$, and which can therefore be applied before $\alpha_2$ has been determined. The patterns which are built later are confined to intervals of total length much smaller than $|A_L|$. Finally, the boundary of the window is a countable set whose position depends on parameters which are chosen at the very end of the construction, after $\balph$ is known. Since the numbers $\langle\balph,m\rangle$, $m\in\Z^2$, form a countable set, these parameters can be chosen so that none of these numbers lies on the boundary of the window modulo $1$, which gives part (i) of Theorem \ref{thm.Main}.

We have restricted our attention to planes in $\R^3$. The Burago--Kleiner argument works in all dimensions $d\ge 2$, and it seems plausible that our method extends to other cut-and-project sets with one-dimensional internal space, but we have not checked the details. For higher-dimensional internal spaces the correspondence between windows and strip families is less simple.   If $\alpha_1$ is badly approximable (for example $\alpha_1=\sqrt{2}$), then a short Borel--Cantelli argument shows that, for every choice of $\alpha_2$, the plane $V_\balph$ contains a vector which is Diophantine in the sense of \cite{HaynKellWeis2014}, so that $V_\balph$ is a Diophantine subspace in their terminology. The proof of \cite[Theorem 1.1]{HaynKellWeis2014} applies to every Diophantine subspace, and it follows that the boundary of our window necessarily has upper Minkowski dimension equal to $1$. This can also be proven directly from our construction. %Finally, Cortez and Navas showed that every Delone set in the hull of their example is non-rectifiable \cite[Remark 14]{CortNava2016}. We do not know whether the same is true for $Y_\balph(W)$, and in particular we do not know whether it holds for every translate of the window.

This paper is organized as follows: In Section \ref{sec.NormalForm} we reduce Theorem \ref{thm.Main} to a statement about a subset of $\Z^2$. In Section \ref{sec.Stretch} we prove the stretched-pair lemma for BL homeomorphisms and use it to show that no BL homeomorphism of the plane can match the area pattern of our configurations. In Section \ref{sec.Finite} we transfer this conclusion to Delone sets, proving Proposition \ref{prop.Finite}. Section \ref{sec.Counting} contains the counting lemmas, from which we deduce Corollary \ref{cor.Encode}, and the construction of strip families with prescribed normals. Finally, in Section \ref{sec.Proof} we construct the window and complete the proofs.

%\noindent\textbf{Acknowledgements:}

\noindent\textbf{Notation:} For $x,y\in\R^n$ we write $\langle x,y\rangle$ for the standard inner product and $|x|$ for the Euclidean norm, and for $x\in\R$ we write $\|x\|$ for the distance from $x$ to the nearest integer. For a measurable set $A$ we write $|A|$ for its Lebesgue measure (length or area, as appropriate), and for a finite set $A$ we write $\#A$ for its cardinality. A map $f$ between subsets of Euclidean spaces is \textit{$L$-BL} if
\[L^{-1}|x-y|\le|f(x)-f(y)|\le L|x-y|\]
for all $x$ and $y$ in its domain. A set $\Lambda\subseteq\R^2$ is a Delone set if it has \textit{separation} $\inf\{|x-y|:x,y\in\Lambda,~x\not=y\}>0$ and finite \textit{covering radius} $\sup_{u\in\R^2}\inf_{x\in\Lambda}|u-x|$.
For $x\in\R$ and $U\subseteq\R/\Z$ we write $x\in U$ to mean that the image of $x$ in $\R/\Z$ lies in $U$.

\noindent\textbf{Disclosure of AI use:} AI was used in this research, as well as to assist in the preparation of this manuscript. However, the final details of all proofs, as well as every statement, citation, and claim in the paper, has been checked by the author.

\section{Reduction to a problem about subsets of $\Z^2$}\label{sec.NormalForm}

For $\balph\in\R^2$ and a set $W\subseteq\R$ we let
\begin{equation}\label{eqn.LambdaDef}
\Lambda_\balph(W)=\left\{m\in\Z^2 : \langle\balph,m\rangle\in W+\Z\right\}.
\end{equation}
It is helpful to picture this set one row at a time. On the horizontal line $m_2=k$, the set $\Lambda_\balph(W)$ consists of those integers $m_1$ for which the point $m_1\alpha_1+k\alpha_2$, of the orbit of $k\alpha_2$ under rotation by $\alpha_1$, lies in $W$ modulo $1$. In this section we show that, when $W$ is contained in an interval of length less than $1$, the sets $Y_\balph(W)$ and $\Lambda_\balph(W)$ agree up to a linear change of coordinates and a bounded perturbation. Theorem \ref{thm.Main} will then follow from the following result, which is what we will actually prove.

\begin{theorem}\label{thm.Plane}
There exist $\alpha_1,\alpha_2\in\R$, with $1,\alpha_1,$ and $\alpha_2$ linearly independent over $\Q$, and a nonempty compact set $W\subseteq\R$, contained in an interval of length less than $1$, which is the closure of its interior and has countable boundary, with the following properties:
\begin{itemize}[itemsep=5bp,parsep=10bp, topsep=0bp]
\item[(i)] For every $m\in\Z^2$ we have that $\langle\balph,m\rangle\notin\partial W+\Z$, and
\item[(ii)] The set $\Lambda_\balph(W)$ is not BL equivalent to $\Z^2$.
\end{itemize}
\end{theorem}

The proof of Theorem \ref{thm.Plane} occupies Sections \ref{sec.Stretch}-\ref{sec.Proof}. To compare $Y_\balph(W)$ with $\Lambda_\balph(W)$, let $\mr{pr}:V_\balph\rar\R^2$ be the linear isomorphism defined by $\mr{pr}(z)=(z_1,z_2)$. Since $V_\balph$ is the graph of the linear function $(z_1,z_2)\mapsto\alpha_1z_1+\alpha_2z_2$, the map $\mr{pr}$ multiplies areas by $1/|\bn_\balph|$.

\begin{lemma}\label{lem.NormalForm}
Let $\balph\in\R^2$, let $W\subseteq\R$ be a set contained in an interval of length less than $1$, and let $s_W=\sup\{|s|:s\in W\}$. Then there is a bijection $\Psi:Y_\balph(W)\rar\Lambda_\balph(W)$ which satisfies
\[
|\mr{pr}(y)-\Psi(y)|\le s_W\quad\text{for all}\quad y\in Y_\balph(W).
\]
\end{lemma}
\begin{proof}
Write $\bn=\bn_\balph$, $\pi=\pi_\balph$, and $K=\{z\in\Z^3:\langle z,\bn\rangle\in W\}$, so that $Y_\balph(W)=\pi(K)$. For $z\in K$ let $s_z=\langle z,\bn\rangle$, so that $\pi(z)=z-s_z\bn/|\bn|^2$.

First we show that $\pi$ is injective on $K$. If $z,z'\in K$ and $\pi(z)=\pi(z')$, then $z-z'=t\bn$ for some $t\in\R$. Comparing third coordinates shows that $t\in\Z$, and on the other hand $t|\bn|^2=s_z-s_{z'}$, which has absolute value less than $1$. Since $|\bn|\ge 1$, this forces $t=0$.

Next, the map $z\mapsto(z_1,z_2)$ is a bijection from $K$ to $\Lambda_\balph(W)$. To see this, let $m\in\Z^2$. Since $\langle(m_1,m_2,z_3),\bn\rangle=\langle\balph,m\rangle-z_3$, the integers $z_3$ for which $(m_1,m_2,z_3)\in K$ are the integers in the set $\langle\balph,m\rangle-W$. This set is contained in an interval of length less than $1$, so it contains at most one integer, and it contains one if and only if $\langle\balph,m\rangle\in W+\Z$.

Combining these two facts, the map $\Psi(\pi(z))=(z_1,z_2)$ is a well defined bijection from $Y_\balph(W)$ to $\Lambda_\balph(W)$. For $z\in K$ we have that
\[
\mr{pr}(\pi(z))-\Psi(\pi(z))=-\frac{s_z}{|\bn|^2}(\alpha_1,\alpha_2),
\]
and the norm of the right-hand side is at most $s_W$, since $|(\alpha_1,\alpha_2)|\le|\bn|$.
\end{proof}

We remark that the bijection $\Psi$ does not depend on any irrationality assumption on $\balph$.

\noindent\emph{Proof of Theorem \ref{thm.Main}, assuming Theorem \ref{thm.Plane}:} Let $\balph$ and $W$ be as in Theorem \ref{thm.Plane}, and write $\bn=\bn_\balph$. If $z\in\Z^3$ and $m=(z_1,z_2)$, then $\langle z,\bn\rangle=\langle\balph,m\rangle-z_3$. Therefore $\langle z,\bn\rangle\in\partial W$ would imply that $\langle\balph,m\rangle\in\partial W+\Z$, which is ruled out by part (i) of Theorem \ref{thm.Plane}. This proves part (i) of Theorem \ref{thm.Main}.

For part (ii), note that $Y_\balph(W)$ is a model set with physical space $V_\balph$, internal space $V_\balph^\perp$, lattice $\Z^3$, and a window which is compact and has nonempty interior. As explained after the statement of Theorem \ref{thm.Main}, the projection $\pi_\balph$ is injective on $\Z^3$ and the internal projection of $\Z^3$ is dense, so it follows from the general theory of model sets that $Y_\balph(W)$ is a Delone set in $V_\balph$ \cite[Chapter 7]{BaakGrim2013}.

Now suppose, by way of contradiction, that there is a BL bijection from $Y_\balph(W)$ to a lattice $\Gamma\subseteq V_\balph$. The map $\mr{pr}(y)\mapsto\Psi(y)$ is a bijection from $\mr{pr}(Y_\balph(W))$ to $\Lambda_\balph(W)$ with bounded displacement, and both of these sets are uniformly discrete, so this map is BL \cite[Proposition 1.8]{SmilSolo2026}. The set $\mr{pr}(\Gamma)$ is a lattice in $\R^2$, so there is a linear map taking it to $\Z^2$. Composing these maps with $\mr{pr}$ and $\mr{pr}^{-1}$, we obtain a BL bijection from $\Lambda_\balph(W)$ to $\Z^2$, contradicting part (ii) of Theorem \ref{thm.Plane}. This completes the proof of Theorem \ref{thm.Main}, assuming Theorem \ref{thm.Plane}.

\section{Stretched pairs}\label{sec.Stretch}

In this section we prove that no BL homeomorphism of the plane can reproduce a certain hierarchical pattern of areas. The argument is that of Burago and Kleiner \cite[Section 3]{BuraKlei1998} (see Lemmas 3.2 and 3.7 there), and our only real change is to the hypothesis. Burago and Kleiner assume that the Jacobian of the map is equal to a prescribed checkerboard function outside a set of small measure. We will only assume that the images of any two adjacent squares of a checkerboard differ in area by a definite amount.%, which is exactly what can be read off from point counts in Section \ref{sec.Finite}.

Let $e=(1,0)$. For an integer $N\ge 2$ we write
\[
R_N=[0,1]\times[0,1/N],\qquad S_i=\left[\frac{i-1}{N},\frac{i}{N}\right]\times\left[0,\frac{1}{N}\right]\quad(1\le i\le N),
\]
so that $R_N$ is the union of the $N$ squares $S_i$, each of area $N^{-2}$. For an integer $M\ge 1$, the \textit{marked points} of $R_N$ are the points
\[
x^i_{pq}=\left(\frac{p+M(i-1)}{NM},\frac{q}{NM}\right),\qquad 1\le i\le N,~0\le p,q\le M,
\]
which form a square grid of mesh $1/(NM)$ in $R_N$, with $x^i_{Mq}=x^{i+1}_{0q}$. The \textit{rows} of $R_N$ are the horizontal lines $\R\times\{q/(NM)\}$, $0\le q\le M$. If $\Sigma(u)=\xi+\ell u$, with $\xi\in\R^2$ and $\ell>0$, then we refer to $\Sigma(R_N)$, $\Sigma(S_i)$, and $\Sigma(x^i_{pq})$ as a rectangle, its squares, and its marked points, and we say that $\Sigma(R_N)$ is \textit{built on} the segment $\Sigma([0,1]\times\{0\})$. Finally, a pair of points $x\not=y$ is \textit{$A$-stretched} by a map $f$ if $|f(x)-f(y)|\ge A|x-y|$.

\begin{lemma}\label{lem.Stretch}
Let $L\ge 1$ and $c\in(0,1]$. Then there are an integer $M\ge 1$ and numbers $\beta>0$ and $N_0\ge 2$, depending only on $L$ and $c$, with the following property. Suppose that $N\ge N_0$, that $\Sigma(u)=\xi+\ell u$ with $\ell>0$, and that $f:\R^2\rar\R^2$ is an $L$-BL homeomorphism which satisfies
\begin{equation}\label{eqn.AreaGap}
\Big|\,|f(\Sigma(S_{i+1}))|-|f(\Sigma(S_i))|\,\Big|\ge c\,|\Sigma(S_i)|\quad\text{for}\quad 1\le i\le N-1.
\end{equation}
If the pair $\Sigma(0,0),\Sigma(1,0)$ is $A$-stretched by $f$, for some $A>0$, then there are two marked points $x$ and $y$ of $\Sigma(R_N)$, lying on a common row at distance $\ell/(NM)$ from each other, such that the pair $x,y$ is $(1+\beta)A$-stretched by $f$.
\end{lemma}

The integer $M$, which fixes how finely each square is subdivided, is chosen first, and $\beta$ and $N_0$ are then chosen in terms of $L$, $c$, and $M$. The order matters, because the argument below uses a pigeonhole principle over the $M+1$ rows.

\begin{proof}
Let $g(u)=\ell^{-1}f(\Sigma(u))$, which is again an $L$-BL homeomorphism of $\R^2$. For every set $\Omega\subseteq\R^2$ we have that $|g(\Omega)|=\ell^{-2}|f(\Sigma(\Omega))|$ and $|\Sigma(\Omega)|=\ell^2|\Omega|$, so $g$ satisfies \eqref{eqn.AreaGap} with $\Sigma$ replaced by the identity. Also, for $x\not=y$ the stretching ratio of $g$ on the pair $x,y$ is equal to that of $f$ on the pair $\Sigma(x),\Sigma(y)$, and $\Sigma$ maps the marked points and rows of $R_N$ to those of $\Sigma(R_N)$. Therefore, after replacing $f$ by $g$, we may assume that $\Sigma$ is the identity, and we do this from now on. We take
\[
\eta=\frac{c}{128\pi(L+2)},\quad M=\left\lceil\frac{L}{\eta}\right\rceil,\quad\lambda=\left(\frac{\eta}{3L}\right)^2,
\]
and
\[
\beta=\frac{\lambda}{16(M+1)},\quad N_0=\left\lceil\frac{32L^2(M+1)}{\lambda}\right\rceil.
\]
Note that $\beta\le\lambda\le 1$, and that $N_0\ge 2$.

After composing $f$ with an isometry of $\R^2$, which changes neither stretching ratios nor areas, we may assume that $f(0,0)=(0,0)$ and that $f(1,0)=(z,0)$ with $z>0$. By hypothesis $z\ge A$, and since $f$ is $L$-BL we have that $1/L\le z\le L$. It is therefore enough to find marked points $x,y$ as in the statement which are $(1+\beta)z$-stretched. Suppose, by way of contradiction, that
\begin{equation}\label{eqn.NoStretch}
|f(x)-f(y)|<(1+\beta)z|x-y|,
\end{equation}
for all distinct marked points $x,y$ on a common row. Note that if $x$ and $y$ are marked points on a common row then the segment between them is a union of segments between adjacent marked points on that row, so \eqref{eqn.NoStretch} is equivalent to the same inequality for adjacent marked points only. This is why it suffices to contradict \eqref{eqn.NoStretch}.

For $1\le i\le N-1$ and $0\le p,q\le M$ let
\[
v^i_{pq}=f(x^{i+1}_{pq})-f(x^i_{pq}),\qquad\bar v=(z/N,0).
\]
Since $x^{i+1}_{pq}=x^i_{pq}+e/N$, it follows from \eqref{eqn.NoStretch} that $|v^i_{pq}|<(1+\beta)z/N$. We say that $v^i_{pq}$ is \textit{regular} if its first coordinate is greater than $(1-\lambda)z/N$, and that the square $S_i$, for $i\le N-1$, is \textit{regular} if all of the $(M+1)^2$ vectors $v^i_{pq}$ are regular.

\noindent\emph{Step 1: a regular square exists.} Suppose not, and for each $1\le i\le N-1$ choose indices $p_i,q_i$ for which $v^i_{p_iq_i}$ is not regular. Sorting the indices $i$ according to the value of $q_i$ and the parity of $i$, we find $q\in\{0,\ldots,M\}$ and indices $i_1<i_2<\cdots<i_t$, all of the same parity, with $q_{i_j}=q$ for each $j$, and with
\[
t\ge\frac{N-1}{2(M+1)}\ge\frac{N}{4(M+1)}.
\]
Write $p_j=p_{i_j}$, and consider the sequence of points
\[
(0,0),~x^1_{0q},~x^{i_1}_{p_1q},~x^{i_1+1}_{p_1q},~x^{i_2}_{p_2q},~x^{i_2+1}_{p_2q},~\ldots,~x^{i_t}_{p_tq},~x^{i_t+1}_{p_tq},~x^N_{Mq},~(1,0).
\]
All of these points except the first and last lie on the row $q$, and, since $i_{j+1}\ge i_j+2$, the first coordinate of $x^{i_j+1}_{p_jq}$ is at most $M(i_j+1)/(NM)$, which is at most the first coordinate of $x^{i_{j+1}}_{p_{j+1}q}$. %So the points on the row are listed from left to right, the segments between consecutive ones have total length $1$, and the first and last steps are vertical, of length $q/(NM)\le 1/N$.
The first coordinate of $f$ increases by $z$ in total along this sequence, and we split this increase into three parts:
\begin{itemize}[itemsep=5bp,parsep=5bp, topsep=0bp]
\item[(i)] Each of the $t$ steps from $x^{i_j}_{p_jq}$ to $x^{i_j+1}_{p_jq}$ contributes the first coordinate of $v^{i_j}_{p_jq}$, which is at most $(1-\lambda)z/N$, since this vector is not regular.
\item[(ii)] The remaining horizontal pieces join marked points on the row $q$ and have total length $1-t/N$, so by \eqref{eqn.NoStretch} they contribute at most $(1+\beta)z(1-t/N)$.
\item[(iii)] The two vertical pieces contribute at most $L/N$ each, since $f$ is $L$-Lipschitz.
\end{itemize}
Therefore
\[
z\le(1-\lambda)z\frac{t}{N}+(1+\beta)z\left(1-\frac{t}{N}\right)+\frac{2L}{N},
\]
which, after dividing by $z$ and using the fact that $z\ge 1/L$, gives
\[
(\lambda+\beta)\frac{t}{N}\le\beta+\frac{2L^2}{N}.
\]
The left-hand side is at least $\lambda/(4(M+1))$, while our choices of $\beta$ and $N\ge N_0$ make the right-hand side at most $\lambda/(8(M+1))$. This contradiction shows that some square $S_i$ with $i\le N-1$ is regular.

\noindent\emph{Step 2: regular vectors are close to $\bar v$.} Let $v=v^i_{pq}=(v_1,v_2)$ be regular. Then \[(1-\lambda)z/N<v_1\le|v|<(1+\beta)z/N,\] so $|v_1-z/N|\le\lambda z/N$, using $\beta\le\lambda$. Also
\[
v_2^2=|v|^2-v_1^2<\left((1+\beta)^2-(1-\lambda)^2\right)\frac{z^2}{N^2}=(\beta+\lambda)(2+\beta-\lambda)\frac{z^2}{N^2}\le\frac{4\lambda z^2}{N^2}.
\]
It follows that $|v-\bar v|^2\le(\lambda^2+4\lambda)z^2/N^2\le 5\lambda L^2/N^2$, and therefore that
\[
|v-\bar v|\le\frac{3L\sqrt{\lambda}}{N}=\frac{\eta}{N}.
\]

\noindent\emph{Step 3: the images of a regular square and the next square have nearly equal area.} Let $S_i$ be a regular square, guaranteed by Step 1, and put $\theta=(\eta+L/M)/N$. We will show that
\begin{equation}\label{eqn.AreaClose}
\Big|\,|f(S_{i+1})|-|f(S_i)|\,\Big|\le 16\pi\theta\left(\frac{L}{N}+\theta\right).
\end{equation}
Parametrize $\partial S_i$ counterclockwise, and define two closed curves on it by $\gamma_0(x)=f(x)$ and $\gamma_1(x)=f(x+e/N)-\bar v$. Since translation by $e/N$ maps $\partial S_i$ onto $\partial S_{i+1}$, preserving orientation, $\gamma_1$ traces the boundary of $f(S_{i+1})-\bar v$. The marked points on each side of $S_i$ are spaced $1/(NM)$ apart, so every $x\in\partial S_i$ is within $1/(2NM)$ of a marked point $x^i_{pq}\in\partial S_i$, and then $x+e/N$ is within $1/(2NM)$ of $x^{i+1}_{pq}$. Using Step 2, it follows that
\begin{align*}
|\gamma_1(x)-\gamma_0(x)|&\le\left|f(x+e/N)-f(x^{i+1}_{pq})\right|+|v^i_{pq}-\bar v|+\left|f(x^i_{pq})-f(x)\right|\\
&\le\frac{L}{2NM}+\frac{\eta}{N}+\frac{L}{2NM}=\theta.
\end{align*}

Now we use the fact that $f$ is a homeomorphism of $\R^2$. The curve $f(\partial S_i)$ is a Jordan curve, the set $f(\mr{int}\,S_i)$ is connected and bounded, and the set $f(\R^2\setminus S_i)$ is connected and unbounded. Since these two sets are disjoint and their union is the complement of $f(\partial S_i)$, the set $f(\mr{int}\,S_i)$ is the bounded component of $\R^2\setminus f(\partial S_i)$. The winding number of $\gamma_0$ about a point $y\notin f(\partial S_i)$ is $\pm1$ if $y\in f(\mr{int}\,S_i)$, by the Jordan curve theorem in the form \cite[Proposition 5.20]{Fult1995}, and it is $0$ if $y$ lies in the unbounded component, since it is constant on that component and vanishes far from the curve. The same holds for $\gamma_1$ and the set $f(\mr{int}\,S_{i+1})-\bar v$. If $y$ is at distance more than $\theta$ from $f(\partial S_i)$, then the homotopy $(1-t)\gamma_0+t\gamma_1$, $0\le t\le 1$, does not pass through $y$, so $\gamma_0$ and $\gamma_1$ have the same winding number about $y$ (see \cite[Chapter 3]{Fult1995}). Therefore, for such $y$ we have that $y\in f(S_i)$ if and only if $y\in f(S_{i+1})-\bar v$. In other words, the symmetric difference of $f(S_i)$ and $f(S_{i+1})-\bar v$ lies in the closed $\theta$-neighborhood $\mc{N}_\theta(f(\partial S_i))$ of the curve $f(\partial S_i)$, and
\[
\Big|\,|f(S_{i+1})|-|f(S_i)|\,\Big|\le|\mc{N}_\theta(f(\partial S_i))|.
\]
To estimate $|\mc{N}_\theta(f(\partial S_i))|$, let $M'=\lceil L/(N\theta)\rceil$ and cut each side of $S_i$ into $M'$ segments of length $1/(NM')$. The image of each of the $4M'$ segments has diameter at most $L/(NM')\le\theta$, so it lies in the disc of radius $\theta$ about the image of one of its endpoints. Therefore $\mc{N}_\theta(f(\partial S_i))$ is covered by $4M'$ discs of radius $2\theta$, and
\[
|\mc{N}_\theta(f(\partial S_i))|\le 16\pi M'\theta^2\le 16\pi\theta^2\left(\frac{L}{N\theta}+1\right)=16\pi\theta\left(\frac{L}{N}+\theta\right).
\]
This proves \eqref{eqn.AreaClose}.

\noindent\emph{Conclusion.} Our choices of $\eta$ and $M$ give $L/M\le\eta$, so $N\theta\le 2\eta=c/(64\pi(L+2))\le 1$. Then for the right-hand side of \eqref{eqn.AreaClose} we have that
\[
\frac{16\pi}{N^2}\,(N\theta)\,(L+N\theta)\le\frac{16\pi}{N^2}\cdot\frac{c}{64\pi(L+2)}\cdot(L+1)\le\frac{c}{4N^2}<c|S_i|.
\]
However, this contradicts \eqref{eqn.AreaGap}. Therefore \eqref{eqn.NoStretch} is false, and we conclude that there is a pair of adjacent marked points on a common row which is $(1+\beta)z$-stretched.
\end{proof}

%We remark that the hypothesis \eqref{eqn.AreaGap} is used only once, at the end of the proof, for the single regular square produced by Step 1. Since we have no control over which square this is, the hypothesis is needed for every adjacent pair. Also, nothing in the proof depends on which of the two squares in each pair has the larger image. In Step 3 it matters that the boundary curves of $f(S_i)$ and $f(S_{i+1})-\bar v$ are compared pointwise, through a common parametrization: knowing only that each of the two boundaries lies close to the other, as sets, would not control the difference of the areas.

Next we iterate Lemma \ref{lem.Stretch}. For integers $M\ge 1$, $N\ge 2$, and $J\ge 1$, we define a finite family $\mc{C}=\mc{C}(M,N,J)$ of rectangles by induction on the \textit{level} $j$. The level $1$ rectangle is $R_N$, built on the segment $[0,1]\times\{0\}$. If $R=\xi+\ell R_N$ is a rectangle of level $j<J$, then for every pair of adjacent marked points $x$ and $x+\ell e/(NM)$ of $R$ on a common row (including the bottom and top rows), the rectangle
\[
x+\frac{\ell}{NM}R_N
\]
built on the segment between them is a rectangle of level $j+1$. The family $\mc{C}$ consists of all rectangles of levels $1,\ldots,J$. A rectangle of level $j$ has length $(NM)^{1-j}$, and its squares have side length $N^{-1}(NM)^{1-j}$.

\begin{proposition}\label{prop.NoHomeo}
Let $L\ge 1$ and $c\in(0,1]$, and let $M$, $\beta$, and $N_0$ be as in Lemma \ref{lem.Stretch}. Suppose that $N\ge N_0$ and that $J\ge 1$ satisfies $(1+\beta)^J>L^2$. Then there is no $L$-BL homeomorphism $f:\R^2\rar\R^2$ with the property that, for every rectangle in $\mc{C}(M,N,J)$ and every pair of adjacent squares $S$ and $S'$ of that rectangle,
\begin{equation}\label{eqn.AreaGapC}
\Big|\,|f(S')|-|f(S)|\,\Big|\ge c\,|S|.
\end{equation}
\end{proposition}
\begin{proof}
Suppose that $f$ is such a map. The pair $(0,0),(1,0)$ is $L^{-1}$-stretched by $f$, since $f$ is $L$-BL. Suppose that, for some $1\le j\le J$, we have found a rectangle $R$ of level $j$, built on a segment whose endpoints are $(1+\beta)^{j-1}L^{-1}$-stretched by $f$. By \eqref{eqn.AreaGapC}, the map $f$ satisfies \eqref{eqn.AreaGap} for $R$, so Lemma \ref{lem.Stretch} gives a pair of adjacent marked points of $R$, on a common row, which is $(1+\beta)^jL^{-1}$-stretched. If $j<J$ then the rectangle built on the segment between these points belongs to $\mc{C}$ and has level $j+1$, so we may continue. After $J$ steps we obtain a pair of points which is $(1+\beta)^JL^{-1}$-stretched, and $(1+\beta)^JL^{-1}>L$, contradicting the fact that $f$ is $L$-Lipschitz.
\end{proof}

Proposition \ref{prop.NoHomeo} applies equally to every translated and dilated copy $x+s\mc{C}$ of the configuration, with $x\in\R^2$ and $s>0$. To see this, suppose that an $L$-BL homeomorphism $f$ satisfies \eqref{eqn.AreaGapC} for every pair of adjacent squares of every rectangle in $x+s\mc{C}$. Then the map
\[
g(u)=s^{-1}f(x+su)
\]
is again an $L$-BL homeomorphism of $\R^2$, and for every square $S$ of a rectangle in $\mc{C}$ we have that
\[
|g(S)|=s^{-2}|f(x+sS)|.
\]
Since $|x+sS|=s^2|S|$, it follows that $g$ satisfies \eqref{eqn.AreaGapC} for every pair of adjacent squares of every rectangle in $\mc{C}$, which is impossible. In what follows the integer $N$ will also be required to be even, which is harmless, since the proposition holds for every $N\ge N_0$.

\section{From homeomorphisms to Delone sets}\label{sec.Finite}

In this section we transfer Proposition \ref{prop.NoHomeo} from homeomorphisms to Delone sets, by replacing the area of the image of a square with the number of points of the set in a large dilate of that square. The idea is that, viewed from far away, an $L$-BL bijection from a Delone set $\Lambda$ to $\Z^2$ looks like an $L$-BL homeomorphism $f$ of the plane, and for a square $S$ and a large number $s$, the number of points of $\Lambda$ in $sS$, divided by $s^2$, looks like the area of $f(S)$. Therefore, if the numbers of points of $\Lambda$ in adjacent squares of a large dilate of $\mc{C}(M,N,J)$ differ by a definite proportion of their area, then the same should be true of the areas of the images under $f$ of the corresponding squares of $\mc{C}(M,N,J)$, which is ruled out by Proposition \ref{prop.NoHomeo}.

To make this precise we use the compactness argument from the proof of \cite[Lemma 2.1]{BuraKlei1998} (see also \cite[Proposition 3.13]{SmilSolo2026}), with one change. Burago and Kleiner apply the argument to a single Delone set, rescaled by a sequence of factors tending to infinity, and they draw their conclusion from the limit. For our construction we need a conclusion at one finite scale, which holds for every set whose counts on adjacent squares differ in this way at that scale, so that the scale can be fixed before the set is known. We therefore argue by contradiction. If there were no such scale, then there would be sets $\Lambda_n$, each admitting an $L$-BL bijection to $\Z^2$, whose counts on adjacent squares of the dilates $s_n\mc{C}(M,N,J)$ differ in this way, for a sequence of scales $s_n$ tending to infinity. These sets may be different for different values of $n$, but the compactness argument still applies to them, and it produces an $L$-BL homeomorphism of the plane which contradicts Proposition \ref{prop.NoHomeo}.

\begin{proposition}\label{prop.Finite}
Let $L\ge 1$ and $c\in(0,1]$, let $M$, $\beta$, and $N_0$ be as in Lemma \ref{lem.Stretch}, and suppose that $N\ge N_0$ and that $J\ge 1$ satisfies $(1+\beta)^J>L^2$. Then, for every $b>0$, there is an $s_*>0$, depending only on $L,c,N,J,$ and $b$, with the following property. Suppose that $\Lambda\subseteq\R^2$ is a set with covering radius at most $b$, and that there are $x\in\R^2$ and $s\ge s_*$ such that, for every rectangle in $x+s\,\mc{C}(M,N,J)$ and every pair of adjacent squares $S$ and $S'$ of that rectangle,
\begin{equation}\label{eqn.CountGap}
\left|\#(\Lambda\cap S')-\#(\Lambda\cap S)\right|\ge c\,|S|.
\end{equation}
Then there is no $L$-BL bijection from $\Lambda$ to $\Z^2$.
\end{proposition}

%Note that no separation hypothesis on $\Lambda$ is needed, since a set which admits an $L$-BL bijection to $\Z^2$ automatically has separation at least $1/L$.

\begin{proof}
Suppose, by way of contradiction, that the proposition fails for some $b>0$. Then, for each $n\in\N$, there are a set $\Lambda_n\subseteq\R^2$ with covering radius at most $b$, a point $x_n\in\R^2$, a number $s_n\ge n$, and an $L$-BL bijection $g_n:\Lambda_n\rar\Z^2$, such that \eqref{eqn.CountGap} holds for $\Lambda_n$ and for every pair of adjacent squares of every rectangle in $x_n+s_n\mc{C}$, where $\mc{C}=\mc{C}(M,N,J)$. Since $g_n$ is $L$-BL, the set $\Lambda_n$ has separation at least $1/L$, so it has only finitely many points in any bounded set, and the infimum in the definition of its covering radius is attained. Therefore we may choose $y_n\in\Lambda_n$ with $|y_n-x_n|\le b$. Replacing $\Lambda_n$ by $\Lambda_n-x_n$ and $g_n$ by $y\mapsto g_n(y+x_n)-g_n(y_n)$, which is again an $L$-BL bijection onto $\Z^2$, we may assume that $x_n=0$, that $|y_n|\le b$, and that $g_n(y_n)=0$.

\noindent\emph{Rescaling and passing to a limit.} Let $X_n=s_n^{-1}\Lambda_n$ and $Z_n=s_n^{-1}\Z^2$, and define $G_n:X_n\rar Z_n$ by $G_n(u)=s_n^{-1}g_n(s_nu)$. Then $G_n$ is an $L$-BL bijection, and we write $H_n:Z_n\rar X_n$ for its inverse. The set $X_n$ has covering radius at most $b/s_n$, the set $Z_n$ has covering radius at most $1/s_n$, and the point $u_n=s_n^{-1}y_n$ satisfies $|u_n|\le b/s_n$, $G_n(u_n)=0$, and $H_n(0)=u_n$. We extend each coordinate function of $G_n$ and $H_n$ to all of $\R^2$ with the same Lipschitz constant $L$, (by the McShane Lipschitz extension formula $\inf_x(\phi(x)+L|u-x|)$, applied to each coordinate $\phi$), we obtain maps $\tilde G_n,\tilde H_n:\R^2\rar\R^2$ which are $\sqrt{2}L$-Lipschitz, which agree with $G_n$ on $X_n$ and with $H_n$ on $Z_n$, and which are uniformly bounded on compact sets, since $\tilde G_n(u_n)=0$ and $\tilde H_n(0)=u_n$. By the Arzel\`{a}--Ascoli theorem and a diagonal argument, after passing to a common subsequence we may assume that $\tilde G_n\rar f$ and $\tilde H_n\rar h$ uniformly on compact sets, for some continuous maps $f,h:\R^2\rar\R^2$.

\noindent\emph{The limit is a BL homeomorphism.} Let $u,u'\in\R^2$, and choose $x_n,x_n'\in X_n$ with $|x_n-u|,|x_n'-u'|\le b/s_n$. Then
\begin{align*}
|G_n(x_n)-f(u)|&\le|\tilde G_n(x_n)-\tilde G_n(u)|+|\tilde G_n(u)-f(u)|\\
&\le\frac{\sqrt{2}Lb}{s_n}+|\tilde G_n(u)-f(u)|\rar 0,
\end{align*}
and similarly $G_n(x_n')\rar f(u')$. Since
\[L^{-1}|x_n-x_n'|\le|G_n(x_n)-G_n(x_n')|\le L|x_n-x_n'|\]
for every $n$, it follows that $f$ is $L$-BL. Next, $G_n(x_n)\in Z_n$, so
\[\tilde H_n(G_n(x_n))=H_n(G_n(x_n))=x_n.\]
Since the maps $\tilde H_n$ are uniformly Lipschitz and converge to $h$ uniformly on compact sets, the left-hand side of the above equation tends to $h(f(u))$, while the right-hand side tends to $u$. Therefore $h\circ f$ is the identity. Choosing points of $Z_n$ close to a given $w\in\R^2$ and arguing in the same way shows that $f\circ h$ is the identity. So $f$ is an $L$-BL homeomorphism of $\R^2$, with inverse $h$.

\noindent\emph{Counts converge to areas.} Let $S$ be a square of one of the rectangles in $\mc{C}$. For $0<\theta\le 1$ let $S_\theta^-$ be the set of points of $S$ at distance at least $\theta$ from $\R^2\setminus S$, and let $S_\theta^+$ be the set of points of $\R^2$ at distance at most $\theta$ from $S$. We will show that
\begin{equation}\label{eqn.CountLimit}
	\lim_{n\rar\infty}\frac{\#(\Lambda_n\cap s_nS)}{s_n^2}=|f(S)|.
\end{equation}
Since $H_n$ is a bijection from $Z_n$ to $X_n$, we have that
\[
\#(\Lambda_n\cap s_nS)=\#(X_n\cap S)=\#\{w\in Z_n:H_n(w)\in S\}.
\]
If $H_n(w)\in S$ then $w=G_n(H_n(w))$, and $|G_n(H_n(w))|\le L(|H_n(w)|+|u_n|)$, so all of the points $w$ in this set lie in a fixed closed disc, which does not depend on $n$. Let $B$ be a closed disc which contains this disc and also the compact set $f(S_1^+)$. Now fix $\theta\in(0,1]$. Since $\tilde H_n\rar h$ uniformly on $B$, for all sufficiently large $n$ we have that $|H_n(w)-h(w)|<\theta$ for all $w\in Z_n\cap B$. For such $n$, if $h(w)\in S_\theta^-$ then $H_n(w)\in S$, and if $H_n(w)\in S$ then $h(w)\in S_\theta^+$. Since $h(w)\in A$ if and only if $w\in f(A)$, and since $f(S_\theta^-)\subseteq B$, we have that
\[
Z_n\cap f(S_\theta^-)\subseteq\{w\in Z_n:H_n(w)\in S\}\subseteq Z_n\cap f(S_\theta^+).
\]
The sets $f(S_\theta^\pm)$ are compact, and their boundaries $f(\partial S_\theta^\pm)$ are images of rectifiable curves under a Lipschitz map, so they have measure zero. For a compact set $K\subseteq\R^2$ whose boundary has measure zero, the number of points of $Z_n=s_n^{-1}\Z^2$ in $K$, multiplied by $s_n^{-2}$, tends to $|K|$. (The squares of side length $s_n^{-1}$ centered at the points of $Z_n$ which meet $\partial K$ have total area tending to $|\partial K|=0$, and all other such squares lie either in $K$ or in its complement.) It follows that
\[
|f(S_\theta^-)|\le\liminf_{n\rar\infty}\frac{\#(\Lambda_n\cap s_nS)}{s_n^2}\le\limsup_{n\rar\infty}\frac{\#(\Lambda_n\cap s_nS)}{s_n^2}\le|f(S_\theta^+)|.
\]
Finally, $f(S_\theta^-)\subseteq f(S)\subseteq f(S_\theta^+)$, and $|f(S_\theta^+)\setminus f(S_\theta^-)|\le 2\pi L^2|S_\theta^+\setminus S_\theta^-|$, which tends to $0$ with $\theta$. This proves \eqref{eqn.CountLimit}.

\noindent\emph{Conclusion.} Applying \eqref{eqn.CountGap} to $\Lambda_n$ and the squares $s_nS$ and $s_nS'$, for adjacent squares $S$ and $S'$ of a rectangle in $\mc{C}$, dividing by $s_n^2$, and using \eqref{eqn.CountLimit}, we obtain
\[
\Big|\,|f(S')|-|f(S)|\,\Big|\ge c\,|S|.
\]
This holds for every pair of adjacent squares of every rectangle in $\mc{C}$, and $f$ is an $L$-BL homeomorphism of $\R^2$, contradicting Proposition \ref{prop.NoHomeo}.
\end{proof}

We emphasize that $s_*$ depends only on $L$, $c$, $N$, $J$, and $b$, and not on the set $\Lambda$. When we construct our window in Section \ref{sec.Proof}, the parameters $c$, $N$, and $J$ will be chosen for each value of $L$ in turn, and the scale of the corresponding copy of $\mc{C}$ will be taken larger than $s_*$ with $b=L$. %We tie $b$ to $L$ in this way because the covering radius of $\Lambda_\balph(W)$ is not known while the construction is in progress. This causes no difficulty, since in the proof of Theorem \ref{thm.Plane} we will only need the configuration for one value of $L$, which can be taken at least as large as this covering radius.

\section{Counting along rows, and strip families}\label{sec.Counting}

In this section we collect the facts about the sets $\{m\in\Z^2:\langle\balph,m\rangle\in U\}$, for $U\subseteq\R/\Z$, which are needed for the construction. The first two lemmas count points of these sets in large convex regions. The key feature of these estimates is that they depend only on $\alpha_1$, and not on $\alpha_2$ or on the position of $U$. This is what will allow us to fix the scales of our construction before $\alpha_2$ and the phases of the window have been chosen. After the counting lemmas we give the proof of Corollary \ref{cor.Encode}, which needs only Lemma \ref{lem.RowCount} and the results of Section \ref{sec.NormalForm}. The last two lemmas describe the plane patterns produced by a particular kind of window.

Our counting estimates are based on the following uniform version of Weyl's equidistribution theorem for irrational rotations (cf. \cite[Chapter 1]{KuipNied1974}). We include the short proof, to make clear that the bound depends only on $\alpha_1$ and $\epsilon$.

\begin{lemma}\label{lem.Weyl}
Let $\alpha_1\in\R\setminus\Q$ and $\epsilon\in(0,1)$. Then there is an integer $N_\epsilon\ge 1$ such that, for every $N\ge N_\epsilon$, every $\zeta\in\R$, and every arc $I\subseteq\R/\Z$,
\[
\left|\#\{0\le j<N:\zeta+j\alpha_1\in I\}-N|I|\right|\le\epsilon N.
\]
\end{lemma}
\begin{proof}
Choose an integer $q\ge 1$ with $0<\|q\alpha_1\|<\epsilon/2$, and write $\kappa=\|q\alpha_1\|$. Splitting $\{0,\ldots,N-1\}$ into its residue classes modulo $q$, each class consists of $K\le N/q+1$ integers $j_0,j_0+q,\ldots,j_0+(K-1)q$, and modulo $1$ the corresponding points $\zeta+j\alpha_1$ form an arithmetic progression with common difference $\pm\kappa$. After a reflection we may assume that the common difference is $\kappa$, and we regard the points as the real numbers $x_0+i\kappa$, $0\le i<K$, which lie in an interval of length $X=(K-1)\kappa$. The intersection of this interval with $I+\Z$ has measure within $|I|\le 1$ of $X|I|$, and it has at most $X+2$ connected components. A component of length $\ell$ contains between $\ell/\kappa-1$ and $\ell/\kappa+1$ points of the progression. It follows that the number of points of the class in $I$ is $(K-1)|I|$, up to an error of at most $1/\kappa+X+2$, and hence it is $K|I|$ up to an error of at most $1/\kappa+K\kappa+3$. Summing over the $q$ classes, the total error is at most $q(1/\kappa+3)+N\kappa$. Since $\kappa<\epsilon/2$, this is at most $\epsilon N$ as soon as $N\ge N_\epsilon=\lceil 2q(1/\kappa+3)/\epsilon\rceil$.
\end{proof}

The next lemma is the main counting estimate. For a finite union of arcs it holds uniformly in $\alpha_2$ and in the position of the arcs. In the statement and below, a \textit{union of $\nu$ arcs} means a union of at most $\nu$ arcs of $\R/\Z$, each of which may be open, closed, or half-open.

\begin{lemma}\label{lem.RowUnif}
Let $\alpha_1\in\R\setminus\Q$, $\epsilon\in(0,1)$, and let $\nu\ge 1$ be an integer. Then there is an $r\ge 1$, depending only on $\alpha_1$, $\epsilon$, and $\nu$, with the following property. For every $\alpha_2\in\R$, every union $U$ of $\nu$ arcs, and every compact convex set $K\subseteq\R^2$ which contains a closed disc of radius $r$, we have that
\[
\left|\#\left\{m\in\Z^2\cap K:\langle\balph,m\rangle\in U\right\}-|U||K|\right|\le\epsilon|K|.
\]
\end{lemma}
\begin{proof}
If $U=\R/\Z$ there is nothing to prove beyond the lattice point estimate at the end of the proof, so we assume that $U\not=\R/\Z$. Then $U$ is a disjoint union of at most $\nu$ arcs (its components). Let $N_1=N_{\epsilon/(4\nu)}$ be the integer given by Lemma \ref{lem.Weyl}, with $\epsilon/(4\nu)$ in place of $\epsilon$. Let $h$ be the height of $K$ (the length of its projection to the vertical axis), and write $n_K=\#(\Z^2\cap K)$.

First we record two geometric facts. Since $K$ contains a disc of radius $r$, one of the two points of $K$ with extreme vertical coordinates is at distance at least $h/2$ from the center of that disc, and the convex hull of the disc and that point has area at least $rh/2$. Also, the perimeter $\mr{per}(K)$ satisfies $\mr{per}(K)\le 2|K|/r$, because $|K|$ is half of the integral, over $\partial K$, of the distance from the center of the disc to the tangent line of $K$, and this distance is at least $r$. So
\begin{equation}\label{eqn.ConvexFacts}
h\le\frac{2|K|}{r}\quad\text{and}\quad\mr{per}(K)\le\frac{2|K|}{r}.
\end{equation}

For $k\in\Z$ let $n_k$ be the number of points of $\Z^2\cap K$ on the horizontal line $m_2=k$. Since $K$ is convex, these points are $(a_k+j,k)$, $0\le j<n_k$, for some integer $a_k$, and their values of $\langle\balph,m\rangle$ are $\zeta+j\alpha_1$, with $\zeta=a_k\alpha_1+k\alpha_2$. If $n_k\ge N_1$, then applying Lemma \ref{lem.Weyl} to each component of $U$ shows that the number of these points with $\langle\balph,m\rangle\in U$ is within $\epsilon n_k/4$ of $|U|n_k$. Call the line $m_2=k$ \textit{short} if $0<n_k<N_1$. The chord of $K$ on a short line has length less than $N_1+1$. The length of the horizontal chord of $K$ at height $y$ is a concave function of $y$, defined on an interval of length $h$, with maximum value at least $2r$. By concavity it is at least $2r$ times the distance from $y$ to the nearer endpoint of this interval, divided by $h$. So the heights of the short lines lie in two intervals of length at most $(N_1+1)h/(2r)$, and there are at most $(N_1+1)h/r+2$ short lines. Each short line contains fewer than $N_1$ points. A point on a short line contributes $0$ or $1$ to the count in the statement, and $|U|\le 1$ to $|U|n_K$, so it changes the difference between these two quantities by at most $1$. Combining these observations with \eqref{eqn.ConvexFacts}, and using $|K|\ge\pi r^2$, we have that
\begin{align*}
\left|\#\left\{m\in\Z^2\cap K:\langle\balph,m\rangle\in U\right\}-|U|n_K\right|&\le\frac{\epsilon n_K}{4}+N_1\left(\frac{(N_1+1)h}{r}+2\right)\\
&\le\frac{\epsilon n_K}{4}+\frac{2N_1(N_1+2)}{r^2}|K|.
\end{align*}

Finally we compare $n_K$ with $|K|$. The unit squares centered at the points of $\Z^2$ which meet $\partial K$ lie in the $\sqrt{2}$-neighborhood of $\partial K$. For a convex set, this neighborhood has area at most $2\sqrt{2}\,\mr{per}(K)+2\pi$, and every other unit square centered at a lattice point lies in $K$ or in its complement. So, by \eqref{eqn.ConvexFacts} and $|K|\ge\pi r^2$,
\[
|n_K-|K||\le 2\sqrt{2}\,\mr{per}(K)+2\pi\le\left(\frac{4\sqrt{2}}{r}+\frac{2}{r^2}\right)|K|\le\frac{8|K|}{r}.
\]
Putting everything together, the difference in the statement is at most
\[
\left(\frac{\epsilon}{4}\left(1+\frac{8}{r}\right)+\frac{2N_1(N_1+2)}{r^2}+\frac{8}{r}\right)|K|,
\]
which is at most $\epsilon|K|$ if $r\ge\max\left(32/\epsilon,\left(8N_1(N_1+2)/\epsilon\right)^{1/2}\right)$.
\end{proof}

The next lemma, which we use to prove Corollary \ref{cor.Encode}, and again in Section \ref{sec.Proof}, follows from Lemma \ref{lem.RowUnif} by approximating $U$ from inside and outside by finite unions of arcs.

\begin{lemma}\label{lem.RowCount}
Let $\alpha_1\in\R\setminus\Q$, $\alpha_2\in\R$, and let $U\subseteq\R/\Z$ be a set with $|\partial U|=0$. Then, for every $\epsilon>0$, there is an $r>0$ such that every compact convex set $K\subseteq\R^2$ which contains a closed disc of radius $r$ satisfies
\[
\left|\#\left\{m\in\Z^2\cap K:\langle\balph,m\rangle\in U\right\}-|U||K|\right|\le\epsilon|K|.
\]
\end{lemma}

\begin{proof}
Assume without loss of generality that $\epsilon<1$. If $\partial U=\emptyset$ then $U$ is empty or all of $\R/\Z$, and the result follows from Lemma \ref{lem.RowUnif} with $\nu=1$. Otherwise, $\partial U$ is a compact set of measure zero, so it is covered by a finite union $O$ of open arcs with $|O|<\epsilon/4$. The closed set $(\R/\Z)\setminus O$ is a finite union of closed arcs, and each of these arcs is connected and disjoint from $\partial U$, so it lies either in the interior of $U$ or in the complement of the closure of $U$. Let $U_-$ be the union of the closed arcs from this collection which lie in the interior of $U$, and let $U_+=U_-\cup O$. Then $U_-\subseteq U\subseteq U_+$, both $U_-$ and $U_+$ are finite unions of arcs, and $|U_+|-|U_-|\le|O|<\epsilon/4$. Applying Lemma \ref{lem.RowUnif} to $U_-$ and $U_+$, with $\epsilon/4$ in place of $\epsilon$, gives the result.
\end{proof}

\noindent\emph{Proof of Corollary \ref{cor.Encode}, assuming Theorem \ref{thm.Plane}:} Let $\balph$ and $W$ be as in Theorem \ref{thm.Plane}, write $Y=Y_\balph(W)$ and $\Lambda=\Lambda_\balph(W)$, let $\Psi$ be as in Lemma \ref{lem.NormalForm}, and let $U\subseteq\R/\Z$ be the image of $W$ in $\R/\Z$, so that $\Lambda=\{m\in\Z^2:\langle\balph,m\rangle\in U\}$. Since $W$ lies in an interval of length less than $1$, we have that $|U|=|W|$, and $|\partial U|=|\partial W|=0$, since $\partial W$ is countable. For $r>0$ and $A\subseteq\R^2$ we write $\mc{N}_r(A)$ for the set of points at distance at most $r$ from $A$.

Let $Q\subseteq V_\balph$ be a closed square of side length $R\ge 1$, and let $P=\mr{pr}(Q)$. Then $P$ is a parallelogram with $|P|=R^2/|\bn_\balph|$, it contains a closed disc of radius $c_1R$, and its perimeter is at most $c_2R$, where $c_1,c_2>0$ depend only on $\balph$. By Lemma \ref{lem.NormalForm}, if $y\in Y$ then $\Psi(y)$ lies within distance $s_W$ of $\mr{pr}(y)$. Therefore, if $y\in Y\cap Q$ and $\Psi(y)\notin P$, then $\Psi(y)\in\mc{N}_{s_W}(\partial P)$. In the same way, if $m\in\Lambda\cap P$ and $\Psi^{-1}(m)\notin Q$, then $m\in\mc{N}_{s_W}(\partial P)$. It follows that
\[
\left|\#(Y\cap Q)-\#(\Lambda\cap P)\right|\le\#\left(\Z^2\cap\mc{N}_{s_W}(\partial P)\right)\le c_3R,
\]
where $c_3$ depends only on $\balph$ and $W$. Now let $\epsilon>0$. By Lemma \ref{lem.RowCount}, applied with $\epsilon/2$ in place of $\epsilon$, for all sufficiently large $R$ we have that
\[
\left|\#(\Lambda\cap P)-|U||P|\right|\le\frac{\epsilon|P|}{2}\le\frac{\epsilon R^2}{2}.
\]
Since $|U||P|=|W|R^2/|\bn_\balph|=\delta R^2$, combining the last two displays proves the first statement of the corollary.

For the second statement, identify $V_\balph$ with $\R^2$ by a linear isometry, and suppose that $r_k$, $Q_k$, and $\varphi_k$ are as in the definition of encoding. If $S\subseteq[0,1]^2$ is a closed axis-parallel square of side length $1/T$, for an integer $T\ge 1$, then, since $\varphi_k$ is a similarity with ratio $r_k$, the set $\varphi_k(S)$ is a closed square of side length $r_k/T$. By what we have just proved,
\[
\mu_k(S)=\frac{\#(Y\cap\varphi_k(S))}{r_k^2}\rar\frac{\delta}{T^2}=\delta|S|\quad\text{as}\quad k\rar\infty.
\]
Since $Y$ is uniformly discrete, the number of points of $Y$ on the boundaries of the squares $\varphi_k(S)$, for $S$ running over the $T^2$ squares of the grid of side length $1/T$ in $[0,1]^2$, is $O_T(r_k)$. Approximating a continuous function on $[0,1]^2$ uniformly by functions which are constant on the open squares of such a grid, it follows that $\mu_k$ converges weakly to $\delta\,\mr{d}u$. Therefore, up to equality almost everywhere, the only density which $Y$ can encode is the constant function $\delta$, which is the Jacobian of the map $u\mapsto\delta^{1/2}u$. This completes the proof of Corollary \ref{cor.Encode}, assuming Theorem \ref{thm.Plane}.

Now we consider the patterns in the plane produced by windows of a particular kind. For an integer $Q\ge 1$, a number $w\in(0,1)$, and $\gamma\in\R$, we write
\begin{align*}
U(Q,w,\gamma)&=\{x\in\R/\Z:\|Qx-\gamma\|<w/2\},\\
U^>(Q,w,\gamma)&=\{x\in\R/\Z:\|Qx-\gamma\|>w/2\}.
\end{align*}
The first of these is a union of $Q$ open arcs of length $w/Q$, equally spaced around the circle, and the second is the interior of its complement. We refer to them, respectively, as the \textit{stage} with modulus $Q$, width $w$, and phase $\gamma$, and its complementary stage. In the plane, for $\bm{\omega}\in\R^2\setminus\{0\}$, we write
\begin{align*}
G(\bm{\omega},w,\gamma)&=\{x\in\R^2:\|\langle\bm{\omega},x\rangle-\gamma\|<w/2\},\\
G^>(\bm{\omega},w,\gamma)&=\{x\in\R^2:\|\langle\bm{\omega},x\rangle-\gamma\|>w/2\}.
\end{align*}
The set $G(\bm{\omega},w,\gamma)$ is a family of parallel open strips with normal vector $\bm{\omega}$, spaced $1/|\bm{\omega}|$ apart, each of width $w/|\bm{\omega}|$, and $G^>(\bm{\omega},w,\gamma)$ is the interior of its complement. We refer to these as \textit{strip families}.

\begin{lemma}\label{lem.Stage}
Let $\balph\in\R^2$, let $Q\ge 1$ be an integer, and write $Q\balph=\bm{\mu}+\bm{\omega}$ with $\bm{\mu}\in\Z^2$ and $\bm{\omega}\in\R^2\setminus\{0\}$. Then, for every $w\in(0,1)$, $\gamma\in\R$, and $m\in\Z^2$, we have that $\langle\balph,m\rangle\in U(Q,w,\gamma)$ if and only if $m\in G(\bm{\omega},w,\gamma)$, and $\langle\balph,m\rangle\in U^>(Q,w,\gamma)$ if and only if $m\in G^>(\bm{\omega},w,\gamma)$.   Consequently, if a set $U\subseteq\R/\Z$ is obtained from finitely many stages and complementary stages by unions and intersections, and $E\subseteq\R^2$ is obtained from the corresponding strip families by the same operations, then \[\{m\in\Z^2:\langle\balph,m\rangle\in U\}=\Z^2\cap E.\]
\end{lemma}
\begin{proof}
We have that $Q\langle\balph,m\rangle=\langle\bm{\mu},m\rangle+\langle\bm{\omega},m\rangle$, and $\langle\bm{\mu},m\rangle\in\Z$, so $\|Q\langle\balph,m\rangle-\gamma\|=\|\langle\bm{\omega},m\rangle-\gamma\|$. This gives the first statement, and the second follows because membership of $m$ is preserved by intersections and unions on both sides.
\end{proof}

The representation $Q\balph=\bm{\mu}+\bm{\omega}$ is not unique, and in our application $\bm{\omega}$ will be chosen to be a short vector, close to a prescribed target. The next lemma shows that, with $\alpha_1$ fixed, $\alpha_2$ can be selected so that this happens for any sequence of targets. The phase $\gamma$ of each stage is a free parameter, which will be chosen at the very end, after $\balph$ is known.

\begin{lemma}\label{lem.Prescribe}
Let $\alpha_1\in\R\setminus\Q$. Suppose that, for $v=1,2,\ldots$ in turn, we are given a vector $\bm{t}_v\in\R^2\setminus\{0\}$ and a number $\epsilon_v\in(0,1)$ with $\epsilon_v|\bm{t}_v|<1$, where these may depend on the choices made at the previous steps. Then integers $Q_v\ge 1$ can be chosen in turn, and there is an $\alpha_2\in \R$, such that $1,\alpha_1,$ and $\alpha_2$ are linearly independent over $\Q$ and, for every $v$, there is a $\bm{\mu}_v\in\Z^2$ with
\[
\left|\bm{\omega}_v-\bm{t}_v\right|\le\epsilon_v|\bm{t}_v|,\quad\text{where}\quad\bm{\omega}_v=Q_v\balph-\bm{\mu}_v.
\]
\end{lemma}
\begin{proof}
Enumerate the countable set $\Q+\Q\alpha_1$ as $\rho_1,\rho_2,\ldots$, and let $I_0=[0,1]$. We will construct nested closed intervals $I_0\supseteq I_1\supseteq I_2\supseteq\cdots$ of positive length. Suppose that $I_{v-1}$ has been constructed, and let $\eta=\epsilon_v|\bm{t}_v|/2$. Choose an integer $Q_v$ with
\[
Q_v\ge\frac{2}{|I_{v-1}|}\quad\text{and}\quad\|Q_v\alpha_1-t_{v,1}\|<\eta,
\]
where $\bm{t}_v=(t_{v,1},t_{v,2})$. The set of $y\in\R$ with $\|Q_vy-t_{v,2}\|\le\eta/2$ is a union of closed intervals of length $\eta/Q_v$, centered at the points $(t_{v,2}+k)/Q_v$, $k\in\Z$, and since $|I_{v-1}|\ge 2/Q_v$, at least one of these intervals lies in $I_{v-1}$. Let $I_v$ be a closed subinterval of it, of positive length, which does not contain $\rho_v$.

Let $\alpha_2$ be the unique point in $\bigcap_vI_v$ (the lengths of the intervals tend to $0$, since $|I_v|\le\eta/Q_v\le 1/Q_v$ and $Q_v\ge 2/|I_{v-1}|$). For each $v$, choose $\mu_{v,1}\in\Z$ with $|Q_v\alpha_1-\mu_{v,1}-t_{v,1}|<\eta$ and $\mu_{v,2}\in\Z$ with $|Q_v\alpha_2-\mu_{v,2}-t_{v,2}|\le\eta/2$. Then $|\bm{\omega}_v-\bm{t}_v|<(5/4)^{1/2}\eta<2\eta=\epsilon_v|\bm{t}_v|$. Finally, $\alpha_2\not=\rho_v$ for every $v$, so $\alpha_2\notin\Q+\Q\alpha_1$, and since $\alpha_1\notin\Q$ this means that $1,\alpha_1,$ and $\alpha_2$ are linearly independent over $\Q$.
\end{proof}

\section{The window: proof of Theorem \ref{thm.Plane}}\label{sec.Proof}

Throughout this section we fix an irrational number $\alpha_1$, and for $m\in\Z^2$ we write $t_m$ for the image of $\langle\balph,m\rangle$ in $\R/\Z$. The window will be the closure of a union of pieces, one for each integer $L\ge 1$. The piece for $L$ is designed so that the set of $m$ with $t_m$ in that piece reproduces, inside a copy of the configuration of Proposition \ref{prop.Finite}, the area pattern which rules out $L$-BL bijections. Before giving the construction we describe the pattern for one value of $L$.

\noindent\emph{Geometry of the configuration.} Let $M\ge 2$ and $N\ge 1000$ be integers, with $N$ even, and let $J\ge 1$ and $s>0$. We consider the dilated configuration $s\,\mc{C}(M,N,J)$ of Section \ref{sec.Stretch}. For $j\ge 1$ let
\[
\ell_j=s(NM)^{1-j}\quad\text{and}\quad\sigma_j=\ell_j/N,
\]
so that the rectangles of level $j$ have length $\ell_j$ and height $\sigma_j$, their squares have side length $\sigma_j$, their marked points are spaced $\ell_{j+1}=\sigma_j/M$ apart, and $\sigma_{j+1}=\sigma_j/(NM)$. We write $\mr{reg}$ for the union of all of the rectangles in $s\,\mc{C}(M,N,J)$.

\begin{lemma}\label{lem.Config}
With the notation above, the following statements hold.
\begin{itemize}[itemsep=5bp,parsep=5bp, topsep=0bp]
\item[(i)] Every square of a rectangle of level $j$ has the form $[k\sigma_j,(k+1)\sigma_j]\times[y,y+\sigma_j]$, with $k\in\Z$ and $y\in\ell_j\Z$, and if it is the $i$th square of its rectangle then $i\equiv k+1~\mathrm{mod}~2$. The rows of a rectangle of level $j$ lie on lines $\{x_2\in\ell_{j+1}\Z\}$.
\item[(ii)] Every rectangle of level $j'>j$ lies in the set $\{x\in\R^2:\mr{dist}(x_2,\ell_{j+1}\Z)\le 1.001\,\sigma_{j+1}\}$.
\item[(iii)] Every $x\in\mr{reg}$ satisfies $|x|\le 1.001\,s$.
\end{itemize}
\end{lemma}
\begin{proof}
	By induction on $j$, every rectangle of level $j$ has its lower left corner in $\ell_j\Z\times\ell_j\Z$. This holds for $j=1$, and if it holds for $j$ then the marked points of a rectangle of level $j$ lie in $\ell_{j+1}\Z\times\ell_{j+1}\Z$, since $\ell_j=NM\ell_{j+1}$. Part (i) follows, because the $i$th square of a rectangle with lower left corner $(x_0,y_0)$ has $k=x_0/\sigma_j+i-1$, and $x_0/\sigma_j\in N\Z$ is even.
	
	Each rectangle of level $j+1$ lies horizontally within its parent and sits on one of its rows, and it has height $\sigma_{j+1}$. Therefore a rectangle $R$ of level $j+1$, together with all of the rectangles descended from it, lies within the horizontal extent of $R$, between the line containing the bottom side of $R$ and the line at height
	\[
	\sigma_{j+1}\left(1+(NM)^{-1}+(NM)^{-2}+\cdots\right)\le 1.001\,\sigma_{j+1}
	\]
	above it, using $NM\ge 2000$. Since the bottom side of $R$ lies on a line $x_2\in\ell_{j+1}\Z$, and every rectangle of level $j'>j$ is either such an $R$ or descended from one, this gives (ii). In the same way $\mr{reg}\subseteq[0,s]\times[0,1.001\,\sigma_1]$, and $\sigma_1=s/N\le s/1000$, which gives (iii).
\end{proof}

\noindent\emph{Strip families for one configuration.} Keeping the notation above, suppose that for each $1\le j\le J$ we are given a strip family $\mr{St}_j=G(\bm{\omega},1/2,\gamma)$, and for each $1\le j<J$ a strip family $\mr{Bd}_{j+1}=G(\bm{\omega}',5/(2N),\gamma')$, where (with $e=(1,0)$ and $e^\perp=(0,1)$)
\begin{equation}\label{eqn.Tolerances}
\begin{aligned}
&\left|\bm{\omega}-\frac{e}{2\sigma_j}\right|\le\frac{\sigma_j}{250s}\cdot\frac{1}{2\sigma_j},\qquad\left|\gamma-\frac34\right|\le\frac{1}{1000N},\\
&\left|\bm{\omega}'-\frac{e^\perp}{\ell_{j+1}}\right|\le\frac{\ell_{j+1}}{500Ns}\cdot\frac{1}{\ell_{j+1}},\qquad|\gamma'|\le\frac{1}{1000N}.
\end{aligned}
\end{equation}
The \textit{stripes} $\mr{St}_j$ should be thought of as approximating the union of the interiors of the even-numbered squares of the rectangles of level $j$, and the \textit{bands} $\mr{Bd}_{j+1}$ as thin horizontal neighborhoods of the rows of these rectangles, which contain all of the rectangles of later levels. We write $\overline{\mr{Bd}}_{j+1}=\R^2\setminus\mr{Bd}^>_{j+1}$, where $\mr{Bd}^>_{j+1}=G^>(\bm{\omega}',5/(2N),\gamma')$, and we set $\overline{\mr{Bd}}_{J+1}=\emptyset$ and $\mr{Bd}^>_{J+1}=\R^2$. The \textit{pattern} associated with these families is
\begin{equation}\label{eqn.PatternDef}
E=\bigcup_{j=1}^J\left(\mr{Bd}_2\cap\cdots\cap\mr{Bd}_j\cap\mr{Bd}^>_{j+1}\cap\mr{St}_j\right),
\end{equation}
where for $j=1$ the empty intersection of bands is read as $\R^2$. Figure \ref{fig.Pattern} shows this set near a rectangle of level $1$, for small values of $M$ and $N$.

\begin{figure}[h]
\caption{Schematic of the pattern $E$ near a rectangle of level $1$, drawn with $N=8, M=3,$ and $J=3$ (shaded regions lie in $E$, and dashed lines show squares of the configuration). Top: away from the bands $\mr{Bd}_2$, the stripes $\mr{St}_1$ fill the even-numbered squares, and inside the bands they are replaced by the finer stripes $\mr{St}_2$. Bottom: an enlargement of the boxed region, showing part of one band, where the stripes $\mr{St}_2$ are in turn replaced by $\mr{St}_3$ inside the thinner bands $\mr{Bd}_3$.\\}\label{fig.Pattern}
\centering
\includegraphics[width=0.9\textwidth]{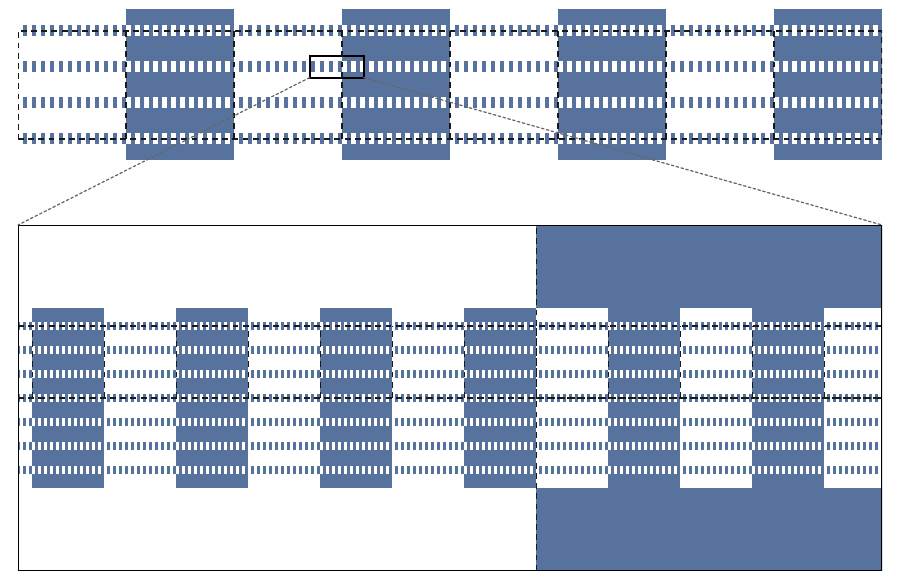}
\end{figure}

\begin{lemma}\label{lem.Pattern}
With the notation above, let $S$ be the $i$th square of a rectangle of level $j$ in $s\,\mc{C}(M,N,J)$, and let $\sigma=\sigma_j$.
\begin{itemize}[itemsep=5bp,parsep=5bp, topsep=0bp]
\item[(i)] Let $S^\circ\subseteq S$ be the set of points of $S$ at distance at least $\sigma/100$ from both vertical sides of $S$, and let $\mr{Sl}\subseteq S$ be the union of the two closed rectangles of width $\sigma/100$ along the vertical sides of $S$. If $i$ is even then $S^\circ\subseteq\mr{St}_j$, and if $i$ is odd then $S\cap\mr{St}_j\subseteq\mr{Sl}$.
\item[(ii)] We have that $S\subseteq\mr{Bd}_2\cap\cdots\cap\mr{Bd}_j$, and, if $j<J$, then $S\cap\overline{\mr{Bd}}_{j+1}$ is contained in a union $\mr{Bnd}$ of at most $M+1$ closed rectangles, each of width $\sigma$ and height between $1.26\,\sigma_{j+1}$ and $2.52\,\sigma_{j+1}$, of total area at most $6|S|/N$.
\item[(iii)] We have that $E\cap S\setminus\overline{\mr{Bd}}_{j+1}=\mr{St}_j\cap S\setminus\overline{\mr{Bd}}_{j+1}$.
\end{itemize}
\end{lemma}
\begin{proof}
By Lemma \ref{lem.Config}(iii), all points $x$ under consideration satisfy $|x|\le 1.001\,s$. For such $x$, the first condition in \eqref{eqn.Tolerances} gives $|\langle\bm{\omega},x\rangle-x_1/(2\sigma_j)|\le 1.001/500$, and together with the condition on $\gamma$ this shows that $\langle\bm{\omega},x\rangle-\gamma$ differs from $x_1/(2\sigma_j)-3/4$ by less than $1/200$. The ideal condition $\|x_1/(2\sigma_j)-3/4\|<1/4$ holds exactly when $x_1$ lies in the interior of an interval $[k\sigma_j,(k+1)\sigma_j]$ with $k$ odd, and if $x_1$ is at distance at least $\sigma_j/100$ from $\sigma_j\Z$ then $x_1/(2\sigma_j)-3/4$ is at distance at least $1/200$ from the boundary values $\pm 1/4+\Z$ of this condition. Therefore such points $x$ lie in $\mr{St}_j$ if and only if $k$ is odd. By Lemma \ref{lem.Config}(i), $k$ is odd exactly when $i$ is even, and this proves (i).

In the same way, the conditions on $\bm{\omega}'$ and $\gamma'$ show that, for $|x|\le 1.001\,s$, the number $\langle\bm{\omega}',x\rangle-\gamma'$ differs from $x_2/\ell_{j+1}$ by less than $1/(100N)$. Since $\ell_{j+1}/N=\sigma_{j+1}$, it follows that points of $\mr{reg}$ with $\mr{dist}(x_2,\ell_{j+1}\Z)<1.24\,\sigma_{j+1}$ lie in $\mr{Bd}_{j+1}$, and that points of $\mr{reg}$ with $\mr{dist}(x_2,\ell_{j+1}\Z)>1.26\,\sigma_{j+1}$ lie in $\mr{Bd}^>_{j+1}$. For $j'<j$, the square $S$ lies in a rectangle of level $j>j'$, so by Lemma \ref{lem.Config}(ii) it lies in $\mr{Bd}_{j'+1}$. This proves the first part of (ii). For the second part, by Lemma \ref{lem.Config}(i) the bottom side of $S$ lies on a line $x_2=y$ with $y\in\ell_j\Z$, and $S$ has height $\sigma_j=M\ell_{j+1}$. So the lines $x_2\in\ell_{j+1}\Z$ which come within $1.26\,\sigma_{j+1}$ of $S$ are the $M+1$ lines $x_2=y+q\ell_{j+1}$, $0\le q\le M$ (the next lines are at distance $\ell_{j+1}=N\sigma_{j+1}$ from $S$). The points of $S$ within $1.26\,\sigma_{j+1}$ of these lines form $M+1$ rectangles of the stated size, and their total area is at most $(M+1)\cdot 2.52\,\sigma_{j+1}\sigma_j\le 6|S|/N$, using $\sigma_{j+1}=\sigma_j/(NM)$.

For (iii), consider the terms in \eqref{eqn.PatternDef}. If $j'<j$, the term with index $j'$ lies in $\mr{Bd}^>_{j'+1}$, which is disjoint from $\mr{Bd}_{j'+1}\supseteq S$. If $j'>j$, the term with index $j'$ lies in $\mr{Bd}_{j+1}\subseteq\overline{\mr{Bd}}_{j+1}$. Finally, on $S\setminus\overline{\mr{Bd}}_{j+1}$ we are inside $\mr{Bd}^>_{j+1}$ and, by (ii), inside $\mr{Bd}_2\cap\cdots\cap\mr{Bd}_j$, so the term with index $j$ coincides there with $\mr{St}_j$.
\end{proof}

\noindent\emph{The construction.} For each integer $L\ge 1$ we choose, in turn, the following data. We put
\[
a_L=\frac{400^{-L}}{4},\qquad c_L=\frac{a_L}{2},\qquad\tau_L=\frac{a_L}{100}.
\]
Let $M_L$, $\beta_L$, and $N_{0,L}$ be given by Lemma \ref{lem.Stretch} for $L$ and $c_L$ and note that, by the proof of that lemma, $M_L\ge 2$. Let $N_L$ be the smallest even integer which is at least $\max(N_{0,L},1000)$, and let $J_L$ be the smallest integer with $(1+\beta_L)^{J_L}>L^2$. Let $s_{*,L}$ be given by Proposition \ref{prop.Finite} for $L,c_L,N_L,J_L,$ and $b=L$. All of these quantities depend only on $L$.

The piece of the window for $L$ will be built from $2J_L-1$ stages, in the sense of Section \ref{sec.Counting}, one for each of the strip families $\mr{St}_1,\mr{Bd}_2,\mr{St}_2,\ldots,\mr{Bd}_{J_L},\mr{St}_{J_L}$ of a pattern \eqref{eqn.PatternDef} with $J=J_L$. We refer to this list of stages as the \textit{hierarchy} $L$. We index the stages of all of the hierarchies by a single integer $v\ge 1$, listing first the stages of hierarchy $1$, then those of hierarchy $2$, and so on, and we write $Q_v$ for the modulus of stage $v$. The moduli will be chosen one at a time, by Lemma \ref{lem.Prescribe}, so when we come to choose the data for hierarchy $L$, the moduli of the stages of hierarchies $1,\ldots,L-1$ are already known. With this in mind, let
\[
\nu_L=1+\sum_{L'<L}\Big(2+\sum_{v}2Q_v\Big),
\]
where the inner sum is over the indices $v$ of the stages of hierarchy $L'$. This number will bound the number of arcs which make up the pieces of the window for the hierarchies $L'<L$, as we will see in the counting argument below.

Let $r_L$ be given by Lemma \ref{lem.RowUnif} for $\alpha_1$, $\epsilon=\tau_L$, and $\nu=\nu_L$. Choose
\[
s_L\ge\max\left(s_{*,L},~200\,r_L\,N_L(N_LM_L)^{J_L-1}\right),
\]
so that the squares of the finest level of $s_L\mc{C}(M_L,N_L,J_L)$ have side length $\sigma_{J_L}\ge 200\,r_L$. With $s=s_L$ and the corresponding numbers $\ell_j$ and $\sigma_j$, we assign to each stage $v$ of hierarchy $L$ a width $w_v$, together with a vector $\bm{t}_v$ and a number $\epsilon_v$ for use in Lemma \ref{lem.Prescribe}. For the stage corresponding to $\mr{St}_j$ we take
\[
w_v=\frac12,\qquad\bm{t}_v=\frac{e}{2\sigma_j},\qquad\epsilon_v=\frac{\sigma_j}{250s_L},
\]
and for the stage corresponding to $\mr{Bd}_{j+1}$ we take
\[
w_v=\frac{5}{2N_L},\qquad\bm{t}_v=\frac{e^\perp}{\ell_{j+1}},\qquad\epsilon_v=\frac{\ell_{j+1}}{500N_Ls_L}.
\]
The construction alternates between the choices above and the steps of Lemma \ref{lem.Prescribe}: once the data for hierarchy $L$ have been chosen, we carry out the steps of Lemma \ref{lem.Prescribe} for the stages of hierarchy $L$, which determines their moduli, and then we proceed to hierarchy $L+1$. For these stages $\epsilon_v|\bm{t}_v|$ is equal to $1/(500s_L)$ or $1/(500N_Ls_L)$, which is less than $1$, and the data of hierarchy $L$ depend only on the moduli of the earlier hierarchies, so the lemma applies. Lemma \ref{lem.Prescribe} then gives a number $\alpha_2$, with $1,\alpha_1,$ and $\alpha_2$ linearly independent over $\Q$, such that, for each $v$, the vector $\bm{\omega}_v=Q_v\balph-\bm{\mu}_v$ satisfies the condition on $\bm{\omega}$ or $\bm{\omega}'$ in \eqref{eqn.Tolerances}.

Now that $\balph$ is fixed, let $T=\{t_m:m\in\Z^2\}$, which is a countable subset of $\R/\Z$. For each stage $v$ of hierarchy $L$, with width $w_v$ and reference phase $\gamma^0_v$ (equal to $3/4$ for stripes and $0$ for bands), choose a phase $\gamma_v$ with $|\gamma_v-\gamma^0_v|\le 1/(1000N_L)$ and with
\begin{equation}\label{eqn.PhaseAvoid}
\gamma_v\notin\left\{\langle\bm{\omega}_v,m\rangle\pm w_v/2:m\in\Z^2\right\}+\Z.
\end{equation}
This is possible because the excluded set is countable. Then choose $p\in\R$ such that none of the points $p$, $p+400^{-L}/2$, and $p+400^{-L}/2+a_L$, for $L\ge 1$, lies in $T$ (modulo $1$), which is possible for the same reason. Finally, define arcs
\[
I_L=\left(p,p+400^{-L}\right]\quad\text{and}\quad A_L=\left(p+\frac{400^{-L}}{2},~p+\frac{400^{-L}}{2}+a_L\right)
\]
(as subsets of $\R/\Z$), let $U_v=U(Q_v,w_v,\gamma_v)$ and $U^>_v=U^>(Q_v,w_v,\gamma_v)$, let $H_L\subseteq\R/\Z$ be the set obtained from the stages of hierarchy $L$ by the same operations as in \eqref{eqn.PatternDef}, and let
\[
W^\circ=\bigcup_{L\ge 1}\left(A_L\cap H_L\right).
\]
We define $W\subseteq\R$ to be the closure of the lift of $W^\circ$ to the interval $(p,p+1/400)$.

Note that $A_L\subseteq I_L\setminus I_{L+1}$, so the arcs $A_L$ are pairwise disjoint, $A_L$ is disjoint from $I_{L+1}$, and $|I_{L+1}|=400^{-L-1}=a_L/100$. Also, by Lemma \ref{lem.Stage}, for each $L$ the set $E_L$ defined by \eqref{eqn.PatternDef} from the strip families $G(\bm{\omega}_v,w_v,\gamma_v)$ of hierarchy $L$ satisfies
\begin{equation}\label{eqn.PieceToPattern}
\{m\in\Z^2:t_m\in A_L\cap H_L\}=\{m\in\Z^2\cap E_L:t_m\in A_L\}.
\end{equation}
By our choice of phases and Lemma \ref{lem.Prescribe}, the strip families of hierarchy $L$ satisfy \eqref{eqn.Tolerances} with $s=s_L$, $M=M_L$, $N=N_L$, and $J=J_L$, so Lemma \ref{lem.Pattern} applies to them.

\noindent\emph{Properties of the window.} Each set $A_L\cap H_L$ is open, and it is a finite union of open arcs, since it is obtained from finitely many arcs by finite intersections and unions. Its boundary lies in the union of $\partial A_L$ and the boundaries of the stages of hierarchy $L$, which are finite sets. Now let $x\in\overline{W^\circ}\setminus W^\circ$, with $x\not=p$. The arcs $A_L$ accumulate only at $p$, so $x$ has a neighborhood which meets only finitely many of them, and therefore $x$ lies in the boundary of one of the sets $A_L\cap H_L$. It follows that $\partial W^\circ$ is contained in the countable set consisting of $p$, the endpoints of the arcs $A_L$, and the boundary points of the stages. Since $W^\circ$ is open, we have that $\partial\overline{W^\circ}\subseteq\partial W^\circ$, so the boundary of $W$ is countable. The set $W$ is compact, it lies in an interval of length less than $1$, and it is the closure of its interior, because $W^\circ$ (lifted) is an open set with closure $W$. That each of the sets $A_L\cap H_L$ is nonempty, and hence that $W$ is nonempty and has infinitely many connected components, will follow from the counts below.

\noindent\emph{Proof of part (i) of Theorem \ref{thm.Plane}.} We show that $T$ does not meet $\partial W^\circ$, which contains the image of $\partial W$ in $\R/\Z$. By the choice of $p$, the set $T$ contains neither $p$ nor any endpoint of any arc $A_L$. A boundary point of the stage $U_v$ is a point $x$ with $Q_vx\in\gamma_v\pm w_v/2+\Z$. If $x=t_m$, then, since $Q_v\langle\balph,m\rangle=\langle\bm{\mu}_v,m\rangle+\langle\bm{\omega}_v,m\rangle$ with $\langle\bm{\mu}_v,m\rangle\in\Z$, we would have that $\gamma_v\in\langle\bm{\omega}_v,m\rangle\mp w_v/2+\Z$, contradicting \eqref{eqn.PhaseAvoid}. This proves part (i). It also shows that
\begin{equation}\label{eqn.LambdaWcirc}
\Lambda_\balph(W)=\{m\in\Z^2:t_m\in W^\circ\}.
\end{equation}

\noindent\emph{Counting points in the squares of one configuration.} Fix $L$, write $a=a_L$, $\tau=\tau_L$, $r=r_L$, and $\Lambda=\Lambda_\balph(W)$, and let \[U^{\mr{bg}}=\bigcup_{L'<L}(A_{L'}\cap H_{L'}).\]
We claim that every square $S$ of a rectangle of level $j$ in $s_L\mc{C}(M_L,N_L,J_L)$ satisfies
\begin{equation}\label{eqn.SquareCount}
\#(\Lambda\cap S)=|U^{\mr{bg}}||S|+\chi_S\,a|S|+e_S,\quad\text{with}\quad|e_S|\le 0.07\,a|S|,
\end{equation}
where $\chi_S=1$ if $S$ is an even-numbered square of its rectangle and $\chi_S=0$ otherwise.

To see this, note first that the pieces $A_{L'}\cap H_{L'}$ with $L'<L$, the piece $A_L\cap H_L$, and the pieces with $L''>L$ lie in the pairwise disjoint sets $U^{\mr{bg}}$, $A_L$, and $I_{L+1}$, respectively. So by \eqref{eqn.LambdaWcirc} and \eqref{eqn.PieceToPattern}, $\#(\Lambda\cap S)=X_1+X_2+X_3$, where
\[
X_1=\#\{m\in\Z^2\cap S:t_m\in U^{\mr{bg}}\},\quad X_2=\#\{m\in\Z^2\cap S\cap E_L:t_m\in A_L\},
\]
and
\[0\le X_3\le\#\{m\in\Z^2\cap S:t_m\in I_{L+1}\}.\] All of the rectangles to which we apply Lemma \ref{lem.RowUnif} below are compact and convex, and contain a closed disc of radius $r$: this holds for $S$ and $S^\circ$, which have sides at least $0.98\,\sigma_j$; for the two rectangles making up $\mr{Sl}$, which have width $\sigma_j/100\ge 2r$; and for the rectangles making up $\mr{Bnd}$, which have height at least $1.26\,\sigma_{j+1}\ge 2r$ (when $j<J_L$). Here we have used that $\sigma_J\ge 200\,r$.

The set $U^{\mr{bg}}$ is open and is not equal to $\R/\Z$, and its boundary lies in the union of the boundaries of the arcs $A_{L'}$ and of the stages of the hierarchies $L'<L$. Each arc $A_{L'}$ has $2$ boundary points, and each stage or complementary stage of modulus $Q_v$ has $2Q_v$, so by the definition of $\nu_L$ the boundary of $U^{\mr{bg}}$ has fewer than $\nu_L$ points. It follows that $U^{\mr{bg}}$ is a union of at most $\nu_L$ arcs. Therefore, by Lemma \ref{lem.RowUnif} with $\epsilon=\tau$,
\[
X_1=|U^{\mr{bg}}||S|\pm\tau|S|\quad\text{and}\quad X_3\le\left(|I_{L+1}|+\tau\right)|S|=\frac{a}{50}|S|.
\]
For $X_2$, write
\[\mr{cnt}(R)=\#\{m\in\Z^2\cap R:t_m\in A_L\}.\]
If $R$ is one of the rectangles $S$ or $S^\circ$, or one of the rectangles making up $\mr{Sl}$ or $\mr{Bnd}$, then, since $A_L$ is a single arc, Lemma \ref{lem.RowUnif} gives $\mr{cnt}(R)=a|R|\pm\tau|R|$. By Lemma \ref{lem.Pattern}(iii), $X_2$ and $\mr{cnt}(S\cap\mr{St}_j)$ count the same points outside $\overline{\mr{Bd}}_{j+1}$, so they differ by at most $\mr{cnt}(\mr{Bnd})$, which by Lemma \ref{lem.Pattern}(ii) is at most $(a+\tau)\cdot 6|S|/N_L\le 0.007\,a|S|$ (and is $0$ if $j=J_L$). If $S$ is even-numbered then, by Lemma \ref{lem.Pattern}(i), $\mr{cnt}(S^\circ)\le\mr{cnt}(S\cap\mr{St}_j)\le\mr{cnt}(S)$, and $|S^\circ|=0.98|S|$, so $\mr{cnt}(S\cap\mr{St}_j)=a|S|\pm(0.02a+\tau)|S|$. If $S$ is odd-numbered then $0\le\mr{cnt}(S\cap\mr{St}_j)\le\mr{cnt}(\mr{Sl})\le(a+\tau)|S|/50\le 0.03\,a|S|$. In either case $X_2=\chi_Sa|S|\pm 0.037\,a|S|$. Adding up the three estimates, and using $\tau=a/100$, proves \eqref{eqn.SquareCount}.

\noindent\emph{Proof of part (ii) of Theorem \ref{thm.Plane}.} First we record the nonemptiness properties mentioned above. Fix $L$, and let $S$ be an even-numbered square of a rectangle of level $1$ in $s_L\mc{C}(M_L,N_L,J_L)$. By the estimate $X_2=a|S|\pm 0.037\,a|S|$ in the proof of \eqref{eqn.SquareCount}, there is an $m\in\Z^2\cap S$ with $t_m\in A_L\cap H_L$, so the set $A_L\cap H_L$ is nonempty. Since the arcs $A_L$ are pairwise disjoint, and each of them is at positive distance from the union of the others, it follows that $W$ has infinitely many connected components. It also follows that $W^\circ$ is not empty, and in particular that $W$ contains an interval $I$ of positive length. Applying Lemma \ref{lem.RowCount} with $U$ equal to the image of $I$ in $\R/\Z$ and $\epsilon=|I|/2$, we find an $r_0>0$ such that every closed disc of radius $r_0$ in $\R^2$ contains a point of $\Lambda_\balph(I)$, and therefore a point of $\Lambda$. So $\Lambda$ is relatively dense, and we let $b_\Lambda$ be its covering radius.

Suppose, by way of contradiction, that there is an $L_0$-BL bijection $g:\Lambda\rar\Z^2$, and choose an integer $L\ge\max(L_0,b_\Lambda)$. Then $g$ is $L$-BL, and $\Lambda$ has covering radius at most $L$. Two adjacent squares $S$ and $S'$ of a rectangle in $s_L\mc{C}(M_L,N_L,J_L)$ have the same area and different parities, so by \eqref{eqn.SquareCount},
\[
\left|\#(\Lambda\cap S')-\#(\Lambda\cap S)\right|\ge a_L|S|-0.14\,a_L|S|>c_L|S|.
\]
Since $s_L\ge s_{*,L}$, Proposition \ref{prop.Finite}, applied with $x=0$, $s=s_L$, and $b=L$, shows that there is no $L$-BL bijection from $\Lambda$ to $\Z^2$. This contradiction proves part (ii) of Theorem \ref{thm.Plane}. Together with the properties of $W$ established above, and the linear independence of $1,\alpha_1,$ and $\alpha_2$, this completes the proof of Theorem \ref{thm.Plane}, and therefore also the proofs of Theorem \ref{thm.Main} and Corollary \ref{cor.Encode}.

We conclude with a remark. The construction gives each hierarchy a fixed contrast $c_L=a_L/2$, which tends to $0$ as $L\rar\infty$. This is unavoidable: the set $\Lambda$ has uniform density by Lemma \ref{lem.RowCount}, so for any fixed $c>0$ the count gaps in \eqref{eqn.CountGap} cannot hold at arbitrarily large scales. The price is that the depth $J_L$ of the configurations grows rapidly with $L$, since $\beta_L$ is of order $c_L^3L^{-6}$ (see the proof of Lemma \ref{lem.Stretch}).

\vspace{.15in}
{\footnotesize
\noindent
Department of Mathematics\\
University of Houston\\
Houston, TX, United States\\
haynes@math.uh.edu
}

\end{document}